\documentclass[12pt]{amsart}

\usepackage{enumerate, amsmath, amsthm, amsfonts, amssymb, xy,  mathrsfs, graphicx, paralist, fancyvrb}
\usepackage[usenames, dvipsnames]{xcolor}
\usepackage[margin=1in]{geometry} 
\usepackage[bookmarks, bookmarksdepth=2, colorlinks=true, linkcolor=blue, citecolor=blue, urlcolor=blue]{hyperref}

\input xy
\xyoption{all}

\numberwithin{equation}{subsection}
\newtheorem{theorem}[equation]{Theorem}

\newtheorem{proposition}[equation]{Proposition}
\newtheorem{lemma}[equation]{Lemma}
\newtheorem{corollary}[equation]{Corollary}
\newtheorem{conjecture}[equation]{Conjecture}

\theoremstyle{definition}
\newtheorem{rmk}[equation]{Remark}
\newenvironment{remark}[1][]{\begin{rmk}[#1] \pushQED{\qed}}{\popQED \end{rmk}}
\newtheorem{eg}[equation]{Example}
\newenvironment{example}[1][]{\begin{eg}[#1] \pushQED{\qed}}{\popQED \end{eg}}
\newtheorem{defn}[equation]{Definition}

\newenvironment{subeqns}[1][]{\addtocounter{equation}{-1}
\begin{subequations}

}{\end{subequations}}

\newcommand{\rA}{\mathrm{A}}

\newcommand{\rB}{\mathrm{B}}

\newcommand{\rC}{\mathrm{C}}

\newcommand{\rD}{\mathrm{D}}

\newcommand{\cE}{\mathcal{E}}

\newcommand{\rE}{\mathrm{E}}

\newcommand{\bF}{\mathbf{F}}

\newcommand{\bG}{\mathbf{G}}

\newcommand{\rH}{\mathrm{H}}

\newcommand{\bL}{\mathbf{L}}
\newcommand{\cL}{\mathcal{L}}

\newcommand{\cM}{\mathcal{M}}

\newcommand{\sM}{\mathscr{M}}

\newcommand{\cN}{\mathcal{N}}

\newcommand{\bO}{\mathbf{O}}
\newcommand{\cO}{\mathcal{O}}

\newcommand{\bP}{\mathbf{P}}

\newcommand{\cQ}{\mathcal{Q}}

\newcommand{\bR}{\mathbf{R}}
\newcommand{\cR}{\mathcal{R}}

\newcommand{\rR}{\mathrm{R}}

\newcommand{\bS}{\mathbf{S}}

\newcommand{\rT}{\mathrm{T}}

\newcommand{\rU}{\mathrm{U}}
\newcommand{\sU}{\mathscr{U}}
\newcommand{\bV}{\mathbf{V}}
\newcommand{\cV}{\mathcal{V}}

\newcommand{\cW}{\mathcal{W}}

\newcommand{\bZ}{\mathbf{Z}}

\newcommand{\rd}{\mathrm{d}}

\newcommand{\fg}{\mathfrak{g}}

\newcommand{\fh}{\mathfrak{h}}

\newcommand{\bk}{\mathbf{k}}

\newcommand{\fn}{\mathfrak{n}}

\newcommand{\fp}{\mathfrak{p}}

\newcommand{\fu}{\mathfrak{u}}

\renewcommand{\phi}{\varphi}
\renewcommand{\emptyset}{\varnothing}
\newcommand{\eps}{\varepsilon}

\renewcommand{\tilde}[1]{\widetilde{#1}}
\newcommand{\ol}[1]{\overline{#1}}

\newcommand{\arxiv}[1]{\href{http://arxiv.org/abs/#1}{{\tt arXiv:#1}}}

\makeatletter
\def\Ddots{\mathinner{\mkern1mu\raise\p@
\vbox{\kern7\p@\hbox{.}}\mkern2mu
\raise4\p@\hbox{.}\mkern2mu\raise7\p@\hbox{.}\mkern1mu}}
\makeatother

\DeclareMathOperator{\coker}{coker}
\renewcommand{\hom}{\operatorname{Hom}}

\DeclareMathOperator{\rank}{rank}
\DeclareMathOperator{\ext}{Ext}

\DeclareMathOperator{\Sym}{Sym}

\DeclareMathOperator{\Tor}{Tor}

\DeclareMathOperator{\Spec}{Spec}
\DeclareMathOperator{\grade}{grade}

\newcommand{\GL}{\mathbf{GL}}
\newcommand{\SL}{\mathbf{SL}}
\newcommand{\Sp}{\mathbf{Sp}}
\newcommand{\SO}{\mathbf{SO}}
\newcommand{\Gr}{\mathbf{Gr}}
\newcommand{\IGr}{\mathbf{IGr}}
\newcommand{\OGr}{\mathbf{OGr}}

\newcommand{\Spin}{\mathbf{Spin}}

\newcommand{\fgl}{\mathfrak{gl}}
\newcommand{\fso}{\mathfrak{so}}
\newcommand{\fsp}{\mathfrak{sp}}

\newcommand{\fosp}{\mathfrak{osp}}

\usepackage[boxsize=.7em]{ytableau}

\newcommand{\osp}{\mathfrak{osp}}
\newcommand{\spo}{\mathfrak{spo}}
\newcommand{\Pin}{\mathbf{Pin}}
\newcommand{\Fl}{\mathbf{Flag}}
\newcommand{\ch}{\mathrm{char}}
\newcommand{\lw}{{\textstyle \bigwedge}}

\title[Orthosymplectic Lie superalgebras and complete intersections]{Orthosymplectic Lie superalgebras, Koszul duality,\\ and a complete intersection analogue of the Eagon--Northcott complex}
\date{September 7, 2016}

\author{Steven V Sam}
\address{Department of Mathematics, University of California, Berkeley}
\curraddr{Department of Mathematics, University of Wisconsin, Madison}
\email{svs@math.wisc.edu}
\urladdr{\url{http://math.wisc.edu/~svs/}}

\thanks{The author was supported by a Miller research fellowship.}

\subjclass[2010]{13D02, % Syzygies, resolutions, complexes
13C40, % Linkage, complete intersections and determinantal ideals
17B10, % Representations, algebraic theory (weights)
17B55, % Homological methods in Lie (super)algebras
18G10% Resolutions; derived functors
}
\begin{document}

\maketitle

\begin{abstract}
We study the ideal of maximal minors in Littlewood varieties, a class of quadratic complete intersections in spaces of matrices. We give a geometric construction for a large class of modules, including all powers of this ideal, and show that they have a linear free resolution over the complete intersection and that their Koszul dual is an infinite-dimensional irreducible representation of the orthosymplectic Lie superalgebra. We calculate the algebra of cohomology operators acting on this free resolution. We prove analogous results for powers of the ideals of maximal minors in the variety of length $2$ complexes when it is a complete intersection, and show that their Koszul dual is an infinite-dimensional irreducible representation of the general linear Lie superalgebra.

This generalizes work of Akin, J\'ozefiak, Pragacz, Weyman, and the author on resolutions of determinantal ideals in polynomial rings to the setting of complete intersections and provides a new connection between representations of classical Lie superalgebras and commutative algebra. As a curious application, we prove that the cohomology of a class of reducible homogeneous bundles on symplectic and orthogonal Grassmannians and $2$-step flag varieties can be calculated by an analogue of the Borel--Weil--Bott theorem.
\end{abstract}

\tableofcontents

\section{Introduction}

In this paper, we are interested on the one hand with representations of orthosymplectic Lie superalgebras (a family of simple Lie superalgebras), and on the other hand, with free resolutions over a class of quadratic complete intersections which we call Littlewood varieties (see \cite{littlewoodcomplexes}). To avoid clumsy technicalities, we will not distinguish between affine varieties and their coordinate rings. The two sides of the story are connected by Koszul duality.

First, it is well-known that quadratic complete intersections, that is, polynomial rings modulo an ideal generated by a regular sequence of quadratic polynomials, are Koszul algebras (this notion and the following notions are reviewed in \S\ref{sec:koszul}). So it has a Koszul dual, which is the universal enveloping algebra of a nilpotent Lie superalgebra. By generalities on Koszul duality, if $R$ is a Koszul $\bk$-algebra and $R^!$ is its Koszul dual, then for any finitely generated graded $R$-module $M$, the space $M^! = \bigoplus_{d \ge 0} \Tor_d^R(M, \bk)^*$ is naturally a finitely generated graded $R^!$-module, which is a direct sum of linear strands.

Generally, Lie superalgebras only have a $\bZ/2$-grading, but many simple Lie superalgebras, including the orthosymplectic Lie superalgebras, have a $\bZ$-grading (not necessarily unique) which lifts its $\bZ/2$-grading. The crucial observation is that for the Littlewood varieties, the nilpotent Lie superalgebra arising from Koszul duality, which has a natural $\bZ$-grading, is a subalgebra of an orthosymplectic Lie superalgebra with a distinguished $\bZ$-grading. 

The main result of this paper is that for a large class of modules $M$ over the coordinate ring of the Littlewood variety, the module $M^!$ consists of one linear strand, and the action of the nilpotent subalgebra on $M^!$ extends to an action of the entire orthosymplectic Lie superalgebra and becomes an irreducible representation. The Littlewood variety sits in a space of matrices, and our class of modules includes all powers of the ideal of maximal minors in this space of matrices. 
Furthermore, the modules $M$ have natural geometric constructions. As a bonus, our results allow us to deduce a Borel--Weil--Bott type theorem for a class of reducible homogeneous bundles over orthogonal and symplectic Grassmannians.

The use of Koszul duality allows one to transform problems about representations of Lie superalgebras into problems about modules over complete intersections, where one can use tools from commutative algebra. We hope that this work will inspire yet further connections. For some context, we mention that the motivations for this paper come in several forms from commutative algebra, invariant theory, and representation theory:
\begin{compactenum}[(1)]
\item The Littlewood varieties were studied in several places under different names (or no name at all) \cite{howe-notes, kac, lovett, lwood}. In \cite{littlewoodcomplexes}, we put the Littlewood varieties into a Lie-theoretic context and show that related free resolutions can be used to recover Littlewood's formulas for branching rules for classical Lie groups (hence the name) and to deduce new formulas for the exceptional Lie groups. In \cite{saturation} they played a central role in proving the saturation property for tensor product coefficients for classical groups. So it is reasonable to look for deeper connections between Littlewood varieties and representation theory.

\item If we consider the polynomial ring which is the coordinate ring of the space of matrices, then all powers of the ideal of maximal minors have a linear free resolution (\cite[Theorem 5.4]{ABW}; we give another proof in \S\ref{sec:EN-powers}). Our results are a perfect analogue in the setting of complete intersections. See \cite{chardin} for a survey of the behavior of homological invariants of powers of an ideal (in the general context of a graded Noetherian ring).

\item The resolution of the ideal of maximal minors is an infinite length generalization of the classical Eagon--Northcott complex \cite{eagonnorthcott} which has been a rich testing ground and source of examples in the theory of finite free resolutions \cite{syzygies, northcott} and its connections with the theory of degeneracy loci \cite{GLP, kempf, schubertcomplexes}.

\item Techniques exist for constructing free resolutions over complete intersections, but little is known about {\it minimal} free resolutions. The complexes in this paper give a large family of concrete examples which should form the beginning of a useful case study.

\item Our results extend work of Akin, J\'ozefiak, Pragacz, Weyman, and the author \cite{aw1, aw2, aw3, JPW, PW, derivedsym} on realizations of finite-dimensional representations of simple Lie superalgebras as syzygies of determinantal ideals. Past results do not cover the important class of orthosymplectic Lie superalgebras; furthermore, our results deal with infinite-dimensional irreducible representations. Aside from this, explicit constructions of representations of Lie superalgebras are generally lacking.

\item Lovett \cite{lovett} investigates the problem of calculating the minimal free resolution of the determinantal ideals in Littlewood varieties, but over the polynomial ring rather than over the coordinate ring of the Littlewood variety. As our results show, even for the maximal minors, the answer over the polynomial ring is complicated while the answer over the complete intersection is dramatically simplified.
\end{compactenum}

\subsection{Outline of paper}

We now describe the results mentioned above in more detail and give an outline of the main results of this paper. The background material for this paper is in \S\ref{sec:prelim}. There is a good analogy between the results we prove for modules over the coordinate rings of Littlewood varieties and modules over the coordinate ring of the whole space of matrices (which is just a polynomial ring). The polynomial case is substantially simpler and as a warmup, we deduce the analogous results in \S\ref{sec:EN-powers} (also they will be used later).

We begin in \S\ref{sec:symp-case} with the symplectic version of the Littlewood variety. Let $V$ be a symplectic vector space and let $E$ be another vector space with $2\dim(E) \le \dim(V)$. Set $A = \Sym(E \otimes V)$ and $B = A / (\lw^2(E))$, where $\lw^2(E)$ is the space of quadratic invariants under the symplectic group $\Sp(V)$. See Example~\ref{eg:typeC-coord} for an explicit example with coordinates. Then $\Spec(A)$ is the space of linear maps $E \to V^*$ and $\Spec(B)$ is the subvariety where the image of $E$ in $V^*$ is an isotropic subspace. Then $B$ is a complete intersection and there is an ideal of maximal minors. 

The main result (Theorem~\ref{thm:Mnu-linear-res}, Remark~\ref{rmk:typeC-minors-explain}) is that all powers of this ideal have a linear free resolution over $B$ and more generally for a class of modules $M$. \S\ref{sec:typeC-prelim} contains basic properties of $B$, and the connection to orthosymplectic Lie superalgebras. In \S\ref{sec:howe}, we describe Howe duality, which is our source of the representations of the orthosymplectic Lie superalgebra, and prove new results about these representations. We prove the main result in \S\ref{sec:main-result-typeC} and give applications. First, we calculate the support varieties of the modules $M$ (Proposition~\ref{prop:support-typeC}). Second, we use Eisenbud's construction for free resolutions over a complete intersection from a free resolution over the polynomial ring \cite{eisenbud-ci} to prove an analogue of the Borel--Weil--Bott theorem for a class of reducible homogeneous bundles on the symplectic Grassmannian (Theorem~\ref{thm:bott-gen}). We give examples in \S\ref{sec:typeC-examples}.

Conjecture~\ref{conj:Ilambda} gives a hint about further directions into which this story should develop. In particular, while the representations we get come from a Howe dual pair, the same will probably not be true for modules associated with lower-order minors. We expect that further understanding in this direction will lead to deeper insight into the structure of the (derived) category of representations of orthosymplectic Lie superalgebras. More importantly, this should lead to an explicit tensor construction of representations, in the same spirit as Schur functors. We will pursue this in future work.

There are analogues of the above results when $V$ is instead a vector space equipped with an orthogonal form, which is discussed in \S\ref{sec:orth-case}. There is also a third case where $V$ is vector space with no extra structure, but one considers pairs of matrices $E \to V$ and $V \to F$ with the condition $\dim(E) + \dim(F) \le \dim(V)$. The locus where the composition $E \to V \to F$ is $0$ is a complete intersection, and analogues of the above results also hold in this case. This is discussed in \S\ref{sec:gl-case}. The main difference is that the orthosymplectic Lie superalgebra is replaced by the general linear Lie superalgebra.

\section{Preliminaries} \label{sec:prelim}

Throughout the paper, we work over a field $\bk$, which we assume is algebraically closed. To avoid this assumption, it is enough to assume that all semisimple Lie algebras are split forms. At the beginning of each section we will state our assumptions on $\ch(\bk)$ if necessary. 

\subsection{Notation}

Given a vector space $V$, its dual is denoted $V^* = \hom_\bk(V,\bk)$. The symmetric, exterior, and divided power algebras are $\Sym(V)$, $\lw^\bullet(V)$, and $\rD(V)$, respectively.

If $V$ is a vector space with $\dim(V) = n$, we set $\det V = \lw^n(V)$. Similarly, if $\cV$ is a vector bundle of rank $n$, we set $\det \cV = \lw^n (\cV)$.

Given representations $V$, $W$ of a group $G$, we write $V \approx W$ if $V$ has a $G$-equivariant filtration whose associated graded is isomorphic to $W$. We only consider the case when $G$ is a reductive algebraic group and $V$ and $W$ are graded representations with finite-dimensional pieces. So when $\ch(\bk)=0$, we can replace $\approx$ by $\cong$ (isomorphism) by complete reducibility.

If $S$ is a graded algebra, $S(-n)$ is the free $S$-module with shifted grading: $S(-n)_i = S_{i-n}$.

\subsection{Koszul duality} \label{sec:koszul}

Let $R$ be a $\bZ_{\ge 0}$-graded $\bk$-algebra such that $R_0 = \bk$ and all $R_i$ are finite-dimensional over $\bk$. If $\bk = R / R_{>0}$ has a linear free resolution over $R$, then $R$ is a {\bf Koszul algebra} \cite[Definition 1.2.1]{BGS}. Let $W$ be the kernel of the multiplication map $R_1 \otimes R_1 \to R_2$. Let $W^\perp$ be the annihilator of $W$ in $R_1^* \otimes R_1^*$ and define $R^! = \rT(R_1^*) / (W^\perp)$ where $\rT(R_1^*)$ is the tensor algebra on $R_1^*$. Then $R^!$ is the {\bf Koszul dual} of $R$ and is also a Koszul algebra \cite[Proposition 2.9.1]{BGS} (and $(R^!)^! = R$). Furthermore, we have an isomorphism of algebras
\begin{align} \label{eqn:koszul-ext}
R^! \cong \ext^\bullet_R(\bk,\bk)^{\rm op}
\end{align}
where the right hand side is equipped with the usual Yoneda product \cite[Theorem 2.10.1]{BGS} and the superscript ${\rm op}$ means we take the opposite ring. Write $E= \ext^\bullet_R(\bk,\bk)$.

Let $M$ be a graded right $R$-module such that $\dim_\bk M_i < \infty$ for all $i$. A linear cochain complex $\beta(M)$ is constructed in \cite[p.105]{BEH} consisting of free right $E$-modules with $\beta(M)^j = M_j \otimes_\bk E(j)$. Via \eqref{eqn:koszul-ext}, this is also a cochain complex of free left $R^!$-modules. If $M$ is a graded left $R$-module, then $M^\vee$ is a graded right $R$-module with $M^\vee_i = M^*_{-i}$, and we define $\bL(M) = \beta(M^\vee)$, which we think of as a chain complex of graded free left $R^!$-modules:
\[
\bL(M): \cdots \to M_i^* \otimes R^!(-i) \to \cdots \to M_1^* \otimes R^!(-1) \xrightarrow{d_1} M_0^* \otimes R^! \xrightarrow{d_0} M_{-1}^* \otimes R^!(1) \to \cdots.
\]
Chasing through the construction, we see that for any $i$, the map $M_i^* \to M_{i-1}^* \otimes R^!_1$ is dual to the multiplication map $R_1 \otimes M_{i-1} \to M_i$. In particular, we can recover $M$ (functorially) from $\bL(M)$ under this finiteness assumption, and so $\bL$ is an equivalence. Let $\bR$ be its inverse. 

We record the above discussion in the next theorem, which is a standard result.

\begin{theorem} \label{thm:dict}
Let $M$ be a finitely generated graded $R$-module. Pick $n$ minimal so that $M_n \ne 0$ and set $N = \coker(M_{n+1}^* \otimes R^!(-n-1) \to M_n^* \otimes R^!(-n))$. 

\begin{compactenum}[\rm (a)]
\item The functor $\bL$ is an equivalence from the category of graded left $R$-modules with finite-dimensional graded pieces to linear complexes of graded finite rank free left $R^!$-modules.
\item For integers $i \ge 0$ and $k$, we have
\[
\rH_k(\bL(M))_{i+k} \cong (\Tor_i^R(\bk, M)_{i+k})^*.
\]
\item $\bL(M)$ is a subcomplex of the minimal free resolution of $N$ if and only if $M$ is generated in degree $n$.
\item $\bL(M)$ is the first linear strand of the minimal free resolution of $N$ if and only if $M$ is generated in degree $n$ and is defined by linear relations over $R$.
\item $\bL(M)$ is the minimal free resolution of $N$ if and only if $M$ is generated in degree $n$ and has a linear free resolution over $R$.
\end{compactenum}
\end{theorem}

\begin{proof}
(a) follows from the discussion above. For (b), we have
\[
\rH_k(\bL(M))_{i+k} = \rH^{-k}(\beta(M^\vee))_{i+k} = \ext_R^{i}(\bk, M^\vee)_{-i-k} = (\Tor^R_i(\bk, M)_{i+k})^*,
\]
where the first equality is by definition of $\bL$, the second is on \cite[p.106]{BEH}, and the third follows from the fact that $\hom_R(\bk, (-)^\vee) \cong (\bk \otimes -)^*$ are isomorphic functors when restricted to finitely generated $R$-modules. One can prove (c), (d), (e) using (b) by following the arguments in \cite[Proof of Theorem 7.7]{syzygies}.
\end{proof}

We end this section with the families of Koszul algebras studied in this paper.

Let $V$ be a vector space. Let $R = \Sym(V) / (f_1, \dots, f_r)$ be a complete intersection where each $f_i$ is a homogeneous polynomial of degree $2$. Define $\fg_1 = V^*$ and let $\fg_2$ be the dual space of $\langle f_1, \dots, f_r \rangle$ with dual basis $f_1^*, \dots, f_r^*$. Each $f_i$ is a quadratic form on $V^*$; let $\beta_i$ be its symmetric bilinear form. Define a Lie superbracket on $\fg = \fg_1 \oplus \fg_2$ by  $[v+w,v'+w'] = \sum_{i=1}^r \beta_i(v,v') f_i^* \in \fg_2$ for $v,v' \in \fg_1$ and $w,w' \in \fg_2$. Define a squaring operation $v^{[2]} = \sum_{i=1}^r f_i(v) f_i^*$ for $v \in \fg_1$ (this is only needed when ${\rm char}(\bk) =2$). The universal enveloping algebra $\rU(\fg)$ is the tensor algebra on $\fg$ modulo the relations 

\begin{compactenum}[\quad (a)]
\item $x \otimes y-(-1)^{\deg(x) \deg(y)} y \otimes x = [x,y]$ for homogeneous $x,y \in \fg$, and 
\item $x \otimes x = x^{[2]}$ for $x \in \fg_1$.  
\end{compactenum}

The bracket and squaring operations preserve the grading on $\fg$, so $\rU(\fg)$ is a graded algebra. The PBW theorem \cite[Theorem 1.32]{chengwang} implies that $\rU(\fg)$ admits a flat degeneration to $\lw^\bullet(\fg_1) \otimes_\bk \Sym(\fg_2)$. In particular, $\rU(\fg)$ is a Noetherian $\bk$-algebra.

\begin{proposition} \label{prop:CI-koszul}
Keep the notation above.

\begin{compactitem}[\rm (a)]
\item $R$ is a Koszul algebra, and $R^! = \rU(\fg)$. 
\item Let $M$ be a finitely generated $R$-module and let $\bG$ be a linear strand of the minimal free resolution of $M$ over $R$. Then $\bR(\bG)$ is a finitely generated $\rU(\fg)$-module.
\end{compactitem}
\end{proposition}

\begin{proof}
(a) follows from \cite[Example 10.2.3]{avramov}.

(b) The algebra $\rU(\fg) \cong \ext^\bullet_R(\bk,\bk)$ acts on $\ext^\bullet_R(M,\bk) = \Tor_\bullet^R(M,\bk)^*$ via the Yoneda product, and this can be identified with the action of $\rU(\fg)$ on $\bigoplus \bR(\bG)$ where the sum is over all linear strands of the minimal free resolution of $M$ (see for example, \cite[p.106]{BEH}). By \cite[Theorem 9.1.4]{avramov} (and the fact that $\Sym(V)$ has finite projective dimension \cite[Theorem 1.1]{syzygies}), the $\ext^\bullet_R(\bk,\bk)$-module $\ext^\bullet_R(M,\bk)$ is finitely generated.
\end{proof}

\subsection{Partitions and Schur functors} \label{sec:schur-functors}

A sequence of integers $\lambda = (\lambda_1, \dots, \lambda_n)$ with $\lambda_1 \ge \cdots \ge \lambda_n \ge 0$ is a {\bf partition}. We write $\ell(\lambda) = \max\{i \mid \lambda_i \ne 0\}$ and $|\lambda| = \sum_i \lambda_i$. If $i > \ell(\lambda)$, we use the convention that $\lambda_i = 0$. 

The sum of two partitions is defined by $(\lambda + \mu)_i = \lambda_i + \mu_i$. The exponential notation $(a^b)$ denotes the number $a$ repeated $b$ times. Its Young diagram is a rectangle, so we also denote this by $b \times a$. We say that $\lambda \subseteq \mu$ if $\lambda_i \le \mu_i$ for all $i$. If $\lambda \subseteq b \times a$, then $(b \times a) \setminus \lambda$ refers to the partition $(a - \lambda_b, \dots, a-\lambda_1)$, i.e., we have rotated the Young diagram of $\lambda$ by $180$ degrees and removed it from the bottom-right corner of the $b \times a$ rectangle. 

The {\bf transpose partition} $\lambda^\dagger$ is defined by $\lambda^\dagger_i = \#\{ j \mid \lambda_j \ge i\}$. This is best explained in terms of Young diagrams, which we define via an example.

\begin{example}
If $\lambda = (5,3,2)$, then $\lambda^\dagger = (3,3,2,1,1)$:
\[
\lambda = \ydiagram{5,3,2}, \qquad \lambda^\dagger = \ydiagram{3,3,2,1,1}.
\]
So $\ell(5,3,2) = 3$ and $|(5,3,2)| = 10$. We have $\lambda = (3 \times 5) \setminus (3,2)$.
\end{example}

Let $\lambda$ be a partition. Then we can define the {\bf Schur functor} $\bS_\lambda$ (this is $L_{\lambda'}$ in \cite[\S 2.1]{weyman}). For any vector space $E$, $\bS_\lambda(E)$ is a representation of the general linear group $\GL(E)$ and $\bS_\lambda(E) \ne 0$ if and only if $\ell(\lambda) \le n$. When $\ch(\bk)=0$, each $\bS_\lambda(E)$ is an irreducible representation of $\GL(E)$. If $\cE$ is a vector bundle, then $\bS_\lambda(\cE)$ is a vector bundle.

If $\dim(E) = n$ and $\lambda = (1^n)$, then $\bS_{(1^n)}(E) = \det E = \lw^n(E)$. Furthermore, $\bS_{\lambda + (1^n)}(E) = \bS_\lambda(E) \otimes (\det E)$. Using this, we can define $\bS_\lambda(E)$ for any weakly decreasing sequence $\lambda$ of integers of length $n$: find $N$ such that $\lambda + (1^N)$ is nonnegative, and define $\bS_\lambda(E) = \bS_{\lambda + (1^N)}(E) \otimes (\det E^*)^N$. This does not depend on the choice of $N$.

\begin{theorem} \label{thm:L-R}
Given partitions $\lambda, \mu, \nu$, let $c^\nu_{\lambda, \mu}$ be the Littlewood--Richardson coefficient (see \cite[\S 9]{macdonald} for details, but we will not need the precise formulation). Then
\[
\bS_\lambda(V) \otimes \bS_\mu(V) \approx \bigoplus_\nu \bS_\nu(V)^{\oplus c^\nu_{\lambda, \mu}}.
\]
If $c^\nu_{\lambda, \mu} \ne 0$, then $|\nu| = |\lambda| + |\mu|$, $\lambda \subseteq \nu$, $\mu \subseteq \nu$, $\nu_1 \le \lambda_1 + \mu_1$, and $\ell(\nu) \le \ell(\lambda) + \ell(\mu)$.
\end{theorem}

\begin{proof}
For the $\approx$, see \cite{boffi}. The last statement is an easy consequence of \cite[\S 9]{macdonald}.
\end{proof}

\begin{theorem}[Cauchy identities] \label{thm:cauchy-id}
Given vector spaces $V$ and $W$, we have
\begin{align*}
\Sym^n(V \otimes W) \approx \bigoplus_{|\lambda| = n} \bS_{\lambda}(V) \otimes \bS_{\lambda}(W), \qquad 
\lw^n(V \otimes W) \approx \bigoplus_{|\lambda| = n} \bS_{\lambda}(V) \otimes \bS_{\lambda^{\dagger}}(W).
\end{align*}
\end{theorem}

\begin{proof}
See \cite[Theorem 2.3.2]{weyman}.
\end{proof}

Finally, let $V$ be a symplectic or orthogonal space, i.e., a vector space equipped with a symplectic or orthogonal form. Let $\lambda$ be a partition with $2\ell(\lambda) \le \dim(V)$. We can define $\bS_{[\lambda]}(V)$, which is a representation of either the symplectic group $\Sp(V)$ or orthogonal group $\bO(V)$, respectively. If $V$ is orthogonal and $2\ell(\lambda) = \dim(V)$, then $\bS_{[\lambda]}(V)$ decomposes, as a representation of $\SO(V)$, into the sum of two non-isomorphic representations which we call $\bS_{[\lambda]^+}(V)$ and $\bS_{[\lambda]^-}(V)$. When $\ch(\bk)=0$, $\bS_\lambda(V)$ is an irreducible representation of $\bO(V)$, and $\bS_{[\lambda]^\pm}(V)$ are irreducible representations of $\SO(V)$. In all other cases, the representation $\bS_{[\lambda]}(V)$ is irreducible when $\ch(\bk)=0$. We refer the reader to \cite[\S\S 17.3, 19.5]{fultonharris} for details when $\ch(\bk)=0$ and to \cite[\S 2]{littlewoodcomplexes} for definitions and basic properties in the general case.

\subsection{Geometric technique for free resolutions} \label{sec:geom}

Let $X$ be a projective variety.  Let
\begin{displaymath}
0 \to \xi \to \eps \to \eta \to 0
\end{displaymath}
be an exact sequence of vector bundles on $X$, with $\eps$ trivial, and let $\cV$ be another vector bundle on $X$.  Put
\begin{displaymath}
A=\rH^0(X; \Sym(\eps)) = \Sym(\rH^0(X; \eps)), \qquad M^{(i)}(\cV) =\rH^i(X; \Sym(\eta) \otimes \cV).
\end{displaymath}
Then each $M^{(i)}(\cV)$ is an $A$-module. For the following, see \cite[\S 5.1]{weyman}.

\begin{theorem} \label{thm:geom-tech}
There is a minimal graded $A$-free complex $\bF_\bullet$ with terms
\[
\bF_i = \bigoplus_{j \ge 0} \rH^j(X; \lw^{i+j}(\xi) \otimes \cV) \otimes A(-i-j)
\]
with the property that for all $i \ge 0$, we have $\rH_{-i}(\bF_\bullet) = M^{(i)}(\cV)$ and $\rH_j(\bF_\bullet) = 0$ for $j>0$. 

In particular, $\bF_i = 0$ for all $i < 0$ if and only if $M^{(j)}(\cV)=0$ for all $j>0$. In this case, $\bF_\bullet$ is a minimal free resolution of $M^{(0)}(\cV)$.
\end{theorem}

\subsection{Borel--Weil theorem} \label{sec:BWB}

Let $\Gr(n,V)$ be the Grassmannian of $n$-dimensional subspaces of a vector space $V$. There is a tautological sequence of vector bundles on $\Gr(n,V)$
\begin{align*}
0 \to \cR \to V \otimes \cO_{\Gr(n,V)} \to \cQ \to 0,
\end{align*}
where $\cR = \{(v, W) \in V \times \Gr(n,V) \mid v \in W\}$ has rank $n$.

\begin{theorem}[Borel--Weil, Kempf] \label{thm:bott}
Let $\lambda,\mu$ be weakly decreasing sequences of integers of lengths $n$ and $\dim(V) - n$, and let $\cV$ be the vector bundle $\bS_{\lambda}(\cR^*) \otimes \bS_{\mu}(\cQ^*)$ on $\Gr(n,V)$. Set $\beta = (\lambda_1, \dots, \lambda_n, \mu_1, \mu_2, \dots)$. 
\begin{compactitem}
\item If $\beta$ is weakly decreasing, then $\rH^0(\Gr(n,V); \cV) = \bS_\beta(V^*)$ as $\GL(V)$-representations, and all higher cohomology of $\cV$ vanishes.
\item Suppose $\lambda_1 = \cdots = \lambda_n = k$ for some $k$ and that $k+1 \le \mu_1 \le k+n$. Then all cohomology of $\cV$ vanishes.
\end{compactitem}
\end{theorem}

\begin{proof}
Let $\pi \colon \Fl(V) \to \Gr(n,V)$ be the projection of the full flag variety on $V$ to $\Gr(n,V)$. Then $\cV = \pi_* \cL$ for a line bundle $\cL$ such that $\rR^i \pi_* \cL = 0$ for $i>0$ (see \cite[Theorem 4.1.12]{weyman}). So we can reduce the first statement to a statement about cohomology of $\cL$ on $\Fl(V)$, and this is contained in \cite[\S II.4]{jantzen}.

For the second statement, let $Y_i$ be the partial flag variety of flags of subspaces of dimensions $i, i+1, \dots, \dim(V)-1$. Name the projections $f \colon \Fl(V) \to Y_n$, $g \colon Y_n \to \Gr(n,V)$, and $h \colon Y_n \to Y_{n+1}$. Set $\cL' = f_* \cL$. Since $\lambda_1 = \cdots = \lambda_n$, $\cL$ is fiberwise trivial along $f$. The fibers of $f$ are connected, so $\cL'$ is a line bundle. The fibers of $h$ are isomorphic to $\bP^n$ and the restriction of $\cL'$ to each $\bP^n$ is isomorphic to $\cO(k-\mu_1)$. By assumption, $0 > k - \mu_1 \ge -n$, so the cohomology is fiberwise trivial and hence globally trivial. So $\rR^i h_* \cL' = 0$ for all $i$. In particular, $\rH^i(Y_n; \cL') = 0$ for all $i$. Since $\pi = gf$, we see that $g_* \cL' = \cV$, and that $\rR^i g_* \cL' = 0$ for $i>0$. We conclude that $\rH^i(\Gr(n,V); \cV) = 0$ for all $i$.
\end{proof}

\begin{remark} \label{rmk:grass-dual}
The Grassmannian of $n$-dimensional subspaces of $V$ and the Grassmannian of  $n$-dimensional quotient spaces of $V^*$ are naturally isomorphic. So in Theorem~\ref{thm:bott}, we could replace $\cR^*$ and $\cQ^*$ by $\cQ$ and $\cR$, respectively, and use $\bS_\beta(V)$ instead of $\bS_\beta(V^*)$.
\end{remark}

Let $\Fl(n,n+d,V)$ be the subvariety of $\Gr(n,V) \times \Gr(n+d,V)$ consisting of pairs of subspaces $(W,W')$ such that $W \subset W'$. Let $\cR_n \subset \cR_{n+d}$ be the corresponding tautological subbundles restricted to $\Fl(n,n+d,V)$.

\begin{theorem}[Borel--Weil, Kempf] \label{thm:bott-flag}
Let $\lambda$, $\mu$, $\nu$ be weakly decreasing sequences of integers of lengths $n$, $d$, $\dim(V)-n-d$, respectively. Let $\cV$ be the vector bundle $\bS_{\lambda}(\cR^*_n) \otimes \bS_{\mu}((\cR_{n+d}/\cR_n)^*) \otimes \bS_{\nu}((V/\cR_{n+d})^*)$ on $\Fl(n,n+d,V)$. 

If $\beta = (\lambda, \mu, \nu)$ is a weakly decreasing sequence, then $\rH^0(\Fl(n,n+d,V); \cV) = \bS_\beta(V^*)$ as $\GL(V)$-representations, and all higher cohomology of $\cV$ vanishes.
\end{theorem}

\begin{proof}
See proof of Theorem~\ref{thm:bott}.
\end{proof}

Now let $V$ be a vector space of dimension $2n$ with symplectic form $\omega_V$. For $d \le n$, $\IGr(d,V)$ is the subvariety of $\Gr(d,V)$ consisting of isotropic subspaces $W$ (i.e., $\omega_V|_W \equiv 0$). Let $\cR$ be the tautological subbundle on $\Gr(d,V)$ restricted to $\IGr(d,V)$. 

\begin{theorem}[Borel--Weil, Kempf] \label{thm:bott-C}
Let $\lambda$ be a partition with $\ell(\lambda) \le d$. As representations of $\Sp(V)$, $\rH^0(\IGr(d,V); \bS_\lambda(\cR^*)) = \bS_{[\lambda]}(V)$ and all higher cohomology of $\bS_{\lambda}(\cR^*)$ vanishes.
\end{theorem}

\begin{proof}
See \cite[(2.1)]{littlewoodcomplexes}.
\end{proof}

Now let $V$ be a vector space of dimension $2n$ or $2n+1$ with orthogonal form $\omega_V$ (if $\ch(\bk)=2$, we need a quadratic form, but the details remain unchanged). For $d \le n$, $\OGr(d,V)$ is the subvariety of $\Gr(d,V)$ of isotropic subspaces $W$ (i.e., $\omega_V|_W \equiv 0$). Let $\cR$ be the tautological subbundle on $\Gr(d,V)$ restricted to $\OGr(d,V)$. When $\dim(V) = 2n$ and $d=n$, $\OGr(d,V)$ has two connected components, which are isomorphic to one another.

\begin{theorem}[Borel--Weil, Kempf] \label{thm:bott-BD}
Let $\lambda$ be a partition with $\ell(\lambda) \le d$. As representations of $\bO(V)$, $\rH^0(\OGr(d,V); \bS_\lambda(\cR^*)) = \bS_{[\lambda]}(V)$ and all higher cohomology of $\bS_\lambda(\cR^*)$ vanishes.
\end{theorem}

\begin{proof}
See \cite[(2.6)]{littlewoodcomplexes} for $\dim(V)$ odd; the even-dimensional case is similar.
\end{proof}

\begin{remark} \label{rmk:bw-functorial}
The results in this section are functorial in $V$. So we can replace $V$ by a vector bundle $\pi \colon \cV \to Z$ on a scheme $Z$, and then $\Gr(n,V)$, $\IGr(n,V)$, etc. are replaced by a relative construction, and cohomology is replaced with higher direct images.
\end{remark}

\subsection{Determinantal modules} \label{sec:EN-powers}

In this section, we construct a family of equivariant modules with a linear free resolution over the coordinate ring of a space of matrices, and derive their basic properties. The modules considered later are generalizations of them. 

Let $E$ and $F$ be vector spaces of dimensions $n$ and $n+d$, respectively ($d \ge 0$). Let $X = \Gr(d,F)$ be the Grassmannian with tautological sub- and quotient bundles $\cR$ and $\cQ$. In the notation of \S\ref{sec:geom}, set $\eta = E \otimes \cQ$ and $\eps = E \otimes F \otimes \cO_X$. Given a partition $\nu$ with $\ell(\nu)\le d$, and given $k \ge \nu_1$, define
\[
\cV^k_\nu = (\det E)^k \otimes (\det \cQ)^k \otimes \bS_\nu \cR.
\]

\begin{proposition} \label{prop:cVnu-vanish}
The higher cohomology of $\Sym(\eta) \otimes \cV^k_\nu$ vanishes.
\end{proposition}

\begin{proof}
By Theorem~\ref{thm:cauchy-id}, $\Sym(\eta)$ has a filtration by terms of the form $\bS_\lambda(E) \otimes \bS_\lambda(\cQ)$. Then the weight of $\bS_\lambda(\cQ) \otimes \cV^k_\nu$ is $(k+\lambda_1, \dots, k+\lambda_n, \nu_1, \dots, \nu_d)$, which is weakly decreasing. Hence all higher cohomology vanishes by Theorem~\ref{thm:bott} and Remark~\ref{rmk:grass-dual}.
\end{proof}

Define the $A = \Sym(E \otimes F)$-module
\begin{equation} \label{eqn:det-module}
\begin{split}
\sM^k_{\nu}(E,F) &= \rH^0(X; \Sym(\eta) \otimes \cV^k_{(d \times k) \setminus \nu})\\
&\approx \bigoplus_{\ell(\lambda) \le n} \bS_{(k+\lambda_1, \dots, k+\lambda_n)}(E) \otimes \bS_{(k + \lambda_1, \dots, k + \lambda_n, k - \nu_d, \dots, k- \nu_1)}(F).
\end{split}
\end{equation}
To prove the $\approx$, use Theorem~\ref{thm:bott} and Remark~\ref{rmk:grass-dual}. The annihilator of this module is $0$. Note that we have used $\cV^k_{(d \times k) \setminus \nu}$ rather than $\cV^k_\nu$.

Still in the notation of \S\ref{sec:geom}, we have $\xi = E \otimes \cR$, so by Theorem~\ref{thm:geom-tech} and Proposition~\ref{prop:cVnu-vanish}, the terms of the minimal free resolution of $\sM^k_{\nu}(E,F)$ are
\begin{align} \label{eqn:sM-res}
\bF^\nu_i = \bigoplus_{j \ge 0} \rH^j(X; \lw^{i+j}(E \otimes \cR) \otimes (\det E)^k \otimes  \bS_{(d \times k) \setminus \nu}(\cR) \otimes (\det \cQ)^k) \otimes A(-i-j).
\end{align}

\begin{proposition} \label{prop:sM-linear}
The resolution of $\sM^k_\nu(E,F)$ is linear. More precisely,
\[
\bF^\nu_i \approx \bigoplus_{\substack{|\lambda|=i\\ \lambda \subseteq n \times d\\ \alpha \subseteq d \times k}} 
(\bS_{(k^n) + \lambda}(E) \otimes \bS_{(k^n, \alpha)}(F))^{\oplus c^\alpha_{\lambda^\dagger, (d \times k) \setminus \nu}} \otimes A(-i).
\]
\end{proposition}

\begin{proof}
Consider \eqref{eqn:sM-res}. Using Theorem~\ref{thm:cauchy-id}, we have
\[
\lw^{i+j}(E \otimes \cR) \approx \bigoplus_{\substack{|\lambda|=i+j\\ \lambda \subseteq n \times d}} \bS_{\lambda}(E) \otimes \bS_{\lambda^\dagger}(\cR),
\]
so we have to show that $\bS_{\lambda^\dagger}(\cR) \otimes \bS_{(d \times k) \setminus \nu}(\cR) \otimes (\det \cQ)^k$ has no higher cohomology when $\lambda \subseteq n \times d$. Let $\bS_\alpha(\cR)$ be a term appearing in the filtration of $\bS_{\lambda^\dagger}(\cR) \otimes \bS_{(d \times k) \setminus \nu}(\cR)$ in Theorem~\ref{thm:L-R} (so with multiplicity $c^\alpha_{\lambda^\dagger, (d \times k) \setminus \nu}$). By assumption, $\lambda^\dagger_1 \le n$, so Theorem~\ref{thm:L-R} implies that $\alpha_1 \le k+n$. To calculate the cohomology of $\bS_\alpha(\cR) \otimes (\det \cQ)^k$, we consider the sequence $(k, \dots, k, \alpha_1, \dots, \alpha_d)$ and use Theorem~\ref{thm:bott} and Remark~\ref{rmk:grass-dual}: if $k \ge \alpha_1$, then $\rH^0(X; \bS_\alpha(\cR) \otimes (\det \cQ)^k) = \bS_{(k^n, \alpha)}(F)$ and there is no higher cohomology; otherwise, $k+1 \le \alpha_1 \le k+n$, and all cohomology vanishes.
\end{proof}

\begin{remark}
If $\nu = (k^d)$, then $\sM^k_{(k^d)}(E,F)$ is the $k$th power of the maximal minors of the generic matrix $E \otimes A(-1) \to F^* \otimes A$. When $k<n$ and $\ch(\bk)=0$, $\bF^{(k^d)}_\bullet$ is also the $k$th linear strand of the ideal of minors of order $n-k+1$ in $A$ \cite[Proposition 6.1.3]{weyman}. 
\end{remark}

\subsection{Modification rules}

In this section, fix a nonnegative integer $k$.

\subsubsection{Type C} \label{sec:typeC-weyl}

We associate to a partition $\lambda$ two quantities, $\iota^\rC_{2k}(\lambda)$ and $\tau^\rC_{2k}(\lambda)$. If $\ell(\lambda) \le k$ we put $\iota^\rC_{2k}(\lambda)=0$ and $\tau^\rC_{2k}(\lambda)=\lambda$.  Suppose $\ell(\lambda)>k$. A {\bf border strip} is a connected skew Young diagram containing no $2 \times 2$ square.  Let $R_{\lambda}$ be the connected border strip of length $2\ell(\lambda)-2k-2$ which starts at the first box in the final row of $\lambda$, if it exists.  If $R_{\lambda}$ exists, is non-empty and $\lambda \setminus R_{\lambda}$ is a partition, then we put $\iota^\rC_{2k}(\lambda)=c(R_{\lambda})+\iota^\rC_{2k}(\lambda \setminus R_{\lambda})$ and $\tau^\rC_{2k}(\lambda)=\tau^\rC_{2k}(\lambda \setminus R_{\lambda})$, where $c(R_{\lambda})$ denotes the number of columns that $R_{\lambda}$ occupies; otherwise we put $\iota^\rC_{2k}(\lambda)=\infty$ and leave $\tau^\rC_{2k}(\lambda)$ undefined.

\begin{example}
Set $k=1$ and $\lambda = (6,5,5,3,2,1,1)$. Then $2\ell(\lambda)-2k-2 = 10$. We shade in the border strip $R_\lambda$ of length $10$ in the Young diagram of $\lambda$:
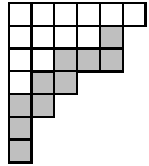
\[
\ytableaushort
{\none, \none \none \none \none \none, \none, \none, \none, \none, \none}
*[*(white)]{6,4,2,1}
*[*(lightgray)]{6,5,5,3,2,1,1} 
\]
In this case $c(R_\lambda) = 5$. If $k=2$, then $\lambda \setminus R_\lambda$ is not a partition, so $\tau_4^\rC(\lambda)$ is undefined.
\end{example}

More details and references can be found in \cite[\S 3.5]{lwood}.

\subsubsection{Type D} \label{sec:typeD-weyl}

We associate to a partition $\lambda$ two quantities $\iota^\rD_{2k}(\lambda)$ and $\tau^\rD_{2k}(\lambda)$. Details and references can be found in \cite[\S 4.4]{lwood}. The definition is the same as the one given in \S \ref{sec:typeC-weyl}, except for two differences (here, as opposed to \cite[\S 4.4]{lwood}, we require $\ell(\tau_{2k}^\rD(\lambda)) \le k$):
\begin{itemize}
\item the border strip $R_{\lambda}$ has length $2\ell(\lambda)-2k$,
\item in the definition of $\iota_{2k}(\lambda)$, we use $c(R_{\lambda})-1$ instead of $c(R_{\lambda})$.
\end{itemize}

\subsubsection{Spin rule} \label{sec:spin-weyl}

We associate to a partition $\lambda$ two quantities $\iota^\Delta_{2k}(\lambda)$ and $\tau^\Delta_{2k}(\lambda)$. Details and references can be found in \cite{spin-cat}. The definition is the same as the one given in \S \ref{sec:typeC-weyl}, except that the border strip $R_{\lambda}$ has length $2\ell(\lambda)-2k-1$.

\subsubsection{Type A} \label{sec:typeA-weyl}

We associate to a pair of partitions $(\lambda, \lambda')$ two quantities, $\iota^\rA_k(\lambda, \lambda')$ and $\tau^\rA_k(\lambda, \lambda')$. Details and references can be found in \cite[\S 5.4]{lwood}.

If $\ell(\lambda)+ \ell(\lambda') \le k$, then $\iota^\rA_k(\lambda, \lambda')=0$ and $\tau^\rA_k(\lambda, \lambda')=(\lambda, \lambda')$. Assume now that $\ell(\lambda) +  \ell(\lambda') > k$. Let $R_{\lambda}$ and $R_{\lambda'}$ be the border strips of length $\ell(\lambda) + \ell(\lambda') - k - 1$ starting in the first box of the final row of $\lambda$ and $\lambda'$, respectively, if they exist. If both $R_{\lambda}$ and $R_{\lambda'}$ exist and are non-empty and both $\lambda \setminus R_{\lambda}$ and $\lambda' \setminus R_{\lambda'}$ are partitions, define
\begin{displaymath}
\iota^\rA_k(\lambda, \lambda')= c(R_{\lambda}) + c(R_{\lambda'})-1 + \iota_k(\lambda \setminus R_{\lambda}, \lambda' \setminus R_{\lambda'})
\end{displaymath}
and $\tau^\rA_k(\lambda, \lambda') = \tau^\rA_k(\lambda \setminus R_{\lambda}, \lambda' \setminus R_{\lambda'})$.  Otherwise, set $\iota_k(\lambda, \lambda') = \infty$ and leave $\tau^\rA_k(\lambda, \lambda')$ undefined.

\subsection{Lie superalgebra homology} \label{sec:liealg-hom}

For background on Lie superalgebras, see \cite[\S 1.1]{chengwang}.
Let $\fg$ be a Lie superalgebra with universal enveloping algebra $\rU(\fg)$, and let $M$ be a $\fg$-module. Define $\rH_i(\fg; M) = \Tor_i^{\rU(\fg)}(M, \bk)$. Given an ideal $\fh \subset \fg$, each $\rH_q(\fh; M)$ is annihilated by $\fh$, and we have the Hochschild--Serre spectral sequence
\begin{align} \label{eqn:HS-SS}
\rE^2_{p,q} = \rH_p(\fg/\fh; \rH_q(\fh; M)) \Rightarrow \rH_{p+q}(\fg; M).
\end{align}
When $\fg$ is a Lie algebra, this is stated in \cite[\S 7.5]{weibel}. It is a special case of the Grothendieck spectral sequence for the derived functors of a composition of two right-exact functors \cite[\S 5.8]{weibel}, so it easily extends to Lie superalgebras.

We now state some calculations of Lie superalgebra homology that will be used.

Let $E|F$ denote a superspace with even part $E$ and odd part $F$. Set $n = \dim(E)$ and $m=\dim(F)$. Let $\fgl(n|m) \cong \fgl(E|F) = (E|F) \otimes (E|F)^*$ be the general linear Lie superalgebra. It has an Abelian subalgebra $\fu = E \otimes F^*$, which is the space of strictly upper-triangular block matrices. So $\rU(\fu) = \bigwedge^\bullet(E \otimes F^*)$.

Given a partition $\lambda$, we can define $\bS_\lambda(E|F)$ (see \cite[\S 2.4]{weyman} where $L_{\lambda'}$ is used in place of $\bS_\lambda$), which is a representation of $\fgl(E|F)$. This is isomorphic to $\bS_{\lambda^\dagger}(F|E)$ and it is nonzero if and only if $\lambda_{n+1} \le m$. We are interested in the homology groups $\rH_i(\fu; \bS_\lambda(E|F))$, which are naturally representations of $\fgl(E) \times \fgl(F)$. Let $\rho = (-1, -2, -3, \dots)$. Given a permutation $w$, set $w \bullet \lambda^\dagger = w(\lambda^\dagger + \rho) - \rho$. Let $W^P$ be the set of finite permutations such that $w \bullet \lambda^\dagger = (\beta_1, \dots, \beta_m, \gamma_1, \gamma_2, \dots)$ where $\beta$ is weakly decreasing and $\gamma$ is a partition. Set $\ell(w) = \#\{i<j \mid w(i)>w(j)\}$.

\begin{proposition} \label{prop:schur-hom}
Assume $\ch(\bk)=0$. With the notation above, we have
\[
\rH_i(E \otimes F^*; \bS_\lambda(E|F)) = \bigoplus_{\substack{w \in W^P\\ \ell(w) = i\\ w\bullet \lambda^\dagger = (\beta, \gamma)}} \bS_\beta(F) \otimes \bS_{\gamma^\dagger}(E).
\]
In particular, setting $\mu = (\lambda^\dagger_{m+1}, \lambda^\dagger_{m+2}, \dots)$, we get
\begin{align*}
\rH_0(E \otimes F^*; \bS_{\lambda}(E|F)) &= \bS_{(\lambda^\dagger_1, \dots, \lambda^\dagger_m)}(F) \otimes \bS_{\mu^\dagger}(E),\\
\rH_1(E \otimes F^*; \bS_{\lambda}(E|F)) &= \bS_{(\lambda^\dagger_1, \dots, \lambda^\dagger_{m-1}, \mu_1 - 1)}(F) \otimes \bS_{(\lambda^\dagger_m+1, \mu_2, \mu_3, \dots)^\dagger}(E),
\end{align*}
so $\bS_\lambda(E|F)$ has a linear presentation as a $\lw^\bullet(E \otimes F^*)$-module if and only if $\lambda_m^\dagger=\lambda_{m+1}^\dagger$.
\end{proposition}

\begin{proof}
See \cite[Corollary 5.1]{CKL}.
\end{proof}

Now let $V$ be a symplectic or orthogonal space. Let $F \subset V$ be a maximal isotropic subspace. The choice of $F$ gives a nilpotent subalgebra $\fn$ as follows:
\begin{compactitem}
\item If $V$ is symplectic, $\fn = \Sym^2(F) \subset \Sym^2(F \oplus F^*) \cong \fsp(V)$. 
\item If $V$ is orthogonal and $\dim(V)$ is even, $\fn = \lw^2(F) \subset \lw^2(F \oplus F^*) \cong \fso(V)$.
\item If $V$ is orthogonal and $\dim(V)$ is odd, $\fn = F \oplus \lw^2(F) \subset \lw^2(F \oplus F^* \oplus \bk) \cong \fso(V)$.
\end{compactitem}
Recall the representations $\bS_{[\lambda]}(V)$ discussed in \S\ref{sec:schur-functors}.

\begin{proposition} \label{prop:par-hom}
Assume $\ch(\bk)=0$. Pick a partition with $\ell(\lambda) \le n = \lfloor \dim(V)/2 \rfloor$. 
\begin{itemize}
\item If $V$ is orthogonal, $\dim(V)$ is even, and $\ell(\lambda) = n$, then we have $\fgl(F)$-equivariant isomorphisms
\[
\rH_0(\fn; \bS_{[\lambda]^+}(V)) = \bS_\lambda(F^*), \qquad \rH_0(\fn; \bS_{[\lambda]^-}(V)) = \bS_{(\lambda_1, \dots, \lambda_{n-1}, -\lambda_n)}(F^*).
\]

\item Otherwise, we have a $\fgl(F)$-equivariant isomorphism
\[
\rH_0(\fn; \bS_{[\lambda]}(V)) = \bS_\lambda(F^*).
\]
\end{itemize}
So a finite-dimensional $\fsp(V)$- or $\fso(V)$-module $M$ is determined by the $\fgl(F)$-module $\rH_0(\fn; M)$.
\end{proposition}

\begin{proof}
The $\rH_0$ calculations are a consequence of the fact that a Cartan subalgebra of $\fgl(F)$ is also a Cartan subalgebra for $\fsp(V)$ or $\fso(V)$, respectively. The last statement follows from the fact that finite-dimensional representations of $\fsp(V)$ or $\fso(V)$ are semisimple.
\end{proof}

Let $V$ be a vector space with a decomposition $V = A \oplus B$ with $\dim(A) = a$ and $\dim(B) = b$. This defines a nilpotent subalgebra $\fn = B^* \otimes A \subset \fgl(V)$. 

\begin{proposition} \label{prop:par-hom-typeA}
Assume $\ch(\bk)=0$. Let $\lambda = (\lambda_1, \dots, \lambda_{a+b})$ be a weakly decreasing sequence of integers. Then
\begin{align*}
\rH_0(B^* \otimes A; \bS_{\lambda}(V)) = \bS_{(\lambda_1, \dots, \lambda_b)}(B) \otimes \bS_{(\lambda_{b+1}, \dots, \lambda_{a+b})}(A).
\end{align*}
So a finite-dimensional $\fgl(V)$-module $M$ is determined by the $\fgl(A) \times \fgl(B)$-module $\rH_0(\fn; M)$.
\end{proposition}

\section{Symplectic Littlewood varieties} \label{sec:symp-case}

In this section, we use $\tau_{2k}$ and $\iota_{2k}$ to denote $\tau^\rC_{2k}$ and $\iota^\rC_{2k}$ (see \S\ref{sec:typeC-weyl}).

\subsection{Preliminaries} \label{sec:typeC-prelim}

Let $E$ be a vector space of dimension $n$ and let $V$ be a symplectic space of dimension $2n+2d$ (we assume $d \ge 0$; some of the results below hold even if $d<0$, but we will not need them). Let $A = \Sym(E \otimes V)$. Using the symplectic form, we have a $\GL(E) \times \Sp(V)$-equivariant inclusion
\[
\lw^2(E) \subset \lw^2(E) \otimes \lw^2(V) \subset \Sym^2(E \otimes V).
\]
Let $B = A / (\lw^2(E))$ be the quotient of $A$ by the ideal generated by $\lw^2(E)$.

\begin{example} \label{eg:typeC-coord}
Let $n=3$ and $d=0$, and pick bases for $E$ and $V$. We think of the coordinates in $A$ as the elements in a $3 \times 6$ matrix:
\[
\begin{pmatrix}
x_{1,1} & x_{1,2} & x_{1,3} & x_{1,4} & x_{1,5} & x_{1,6}\\
x_{2,1} & x_{2,2} & x_{2,3} & x_{2,4} & x_{2,5} & x_{2,6}\\
x_{3,1} & x_{3,2} & x_{3,3} & x_{3,4} & x_{3,5} & x_{3,6}
\end{pmatrix}
\]
If the symplectic form on $V$ is given by $\omega_V(v, w) = v_1w_2 - v_2w_1 + v_3w_4 - v_4w_3 + v_5w_6 - v_6w_5$, then $\lw^2(E)$ is spanned by the $3$ equations
\[
\begin{array}{l}
x_{1,1}x_{2,2} - x_{1,2}x_{2,1} + x_{1,3}x_{2,4} - x_{1,4}x_{2,3} + x_{1,5}x_{2,6} - x_{1,6}x_{2,5},\\
x_{1,1}x_{3,2} - x_{1,2}x_{3,1} + x_{1,3}x_{3,4} - x_{1,4}x_{3,3} + x_{1,5}x_{3,6} - x_{1,6}x_{3,5},\\
x_{2,1}x_{3,2} - x_{2,2}x_{3,1} + x_{2,3}x_{3,4} - x_{2,4}x_{3,3} + x_{2,5}x_{3,6} - x_{2,6}x_{3,5}. 
\end{array} \qedhere
\]
\end{example}

We have $\Spec(A) = \hom(E,V^*)$, the space of linear maps $E \to V^*$. We can identify $\Spec(B) \subset \hom(E,V^*)$ with the set of maps $\phi$ such that the composition $E \xrightarrow{\phi} V^* \cong V \xrightarrow{\phi^*} E^*$ is $0$ (here $V \cong V^*$ is the isomorphism induced by the symplectic form on $V$). Equivalently, $\Spec(B)$ is the subvariety of maps $E \to V^*$ such that the image of $E$ is isotropic, i.e., the symplectic form restricts to the $0$ form on it.

\begin{proposition} \label{prop:typeC-basicfacts}
\begin{compactenum}[\rm (a)]
\item $B$ is an integral domain, i.e., the ideal generated by $\lw^2(E)$ is prime.

\item As a representation of $\GL(E) \times \Sp(V)$, we have
\[
B \approx \bigoplus_{\ell(\lambda) \le n} \bS_\lambda(E) \otimes \bS_{[\lambda]}(V).
\]

\item $\bS_\lambda(E) \otimes \bS_{[\lambda]}(V)$ is in the ideal generated by $\bS_\mu(E) \otimes \bS_{[\mu]}(V)$ if and only if $\lambda \supseteq \mu$.

\item Two closed points in $\Spec(B)$ are in the same $\GL(E) \times \Sp(V)$ orbit if and only if they have the same rank as a matrix. Each orbit closure is a normal variety with rational singularities. In particular, they are Cohen--Macaulay varieties.

\item The ideal $(\lw^2(E))$ is generated by a regular sequence. In particular, $B$ is a complete intersection and is a Koszul algebra.
\end{compactenum}
\end{proposition}

\begin{proof}
For (a), (b), (d), and (e), see \cite[Theorem 2.2]{littlewoodcomplexes}. (c) follows from the analogous statement for $\Sym(E \otimes V)$ \cite[Corollary 4.2]{dEP}. 
\end{proof}

\begin{remark}
Part (e) of Proposition~\ref{prop:typeC-basicfacts} also holds if $\dim(V) = 2\dim(E) - 2$. Essentially the same proof as for \cite[Theorem 2.2(3)]{littlewoodcomplexes} works.
\end{remark}

Let $\IGr(n+d,V)$ be the isotropic Grassmannian of rank $n+d$ isotropic subspaces of $V$ with tautological subbundle $\cR$. Let $\nu \subseteq (k^d)$ be a partition. Using \eqref{eqn:det-module} in a relative situation, we get a $\Sym(E \otimes \cR^*)$-module $\sM^k_\nu(E, \cR^*)$.

\begin{proposition} \label{prop:cM-vanish}
The higher sheaf cohomology groups of $\sM^k_\nu(E, \cR^*)$ vanish and
\[
\rH^0(\IGr(n+d,V); \sM^k_\nu(E, \cR^*)) \approx \bigoplus_{\ell(\lambda) \le n} \bS_{(k^n) + \lambda}(E) \otimes \bS_{[(k^{n})+\lambda, k-\nu_d, \dots, k-\nu_1]}(V).
\]
\end{proposition}

\begin{proof}
By \eqref{eqn:det-module}, we have 
\[
\sM^k_\nu(E,\cR^*) \approx \bigoplus_{\ell(\lambda) \le n} \bS_{(k+\lambda_1, \dots, k+\lambda_n)}(E) \otimes \bS_{(k + \lambda_1, \dots, k + \lambda_n, k - \nu_d, \dots, k- \nu_1)}(\cR^*).
\]
Now the result follows from Theorem~\ref{thm:bott-C}.
\end{proof}

We define 
\begin{equation} \label{eqn:cM-module}
\begin{split}
\cM^k_\nu &= \rH^0(\IGr(n+d,V); \sM^k_\nu(E, \cR^*))\\
&\approx \bigoplus_{\ell(\lambda) \le n} \bS_{(k^n) + \lambda}(E) \otimes \bS_{[((k^n)+\lambda, k-\nu_d, \dots, k-\nu_1)]}(V).
\end{split}
\end{equation}
So $\cM^k_\nu$ is an $A$-module. In fact, the scheme-theoretic image of $\Spec(\Sym(E \otimes \cR^*)) \to \Spec(A)$ is $\Spec(B)$ (this is clear set-theoretically, and then use that both $\Spec(\Sym(E \otimes \cR^*))$ and $\Spec(B)$ are reduced schemes), so $\cM^k_\nu$ is also a $B$-module. 

The following result is the symplectic version of the main result of this paper. A more detailed version of this theorem is contained in Theorem~\ref{thm:main-typeC}, which will be proven in \S\ref{sec:main-result-typeC}.

\begin{theorem} \label{thm:Mnu-linear-res}
Assume $\ch(\bk)=0$. The minimal free resolution of $\cM_\nu^k$ over $B$ is linear.
\end{theorem}

Let $\tilde{V}$ be a $\bZ/2$-graded space (from now on, superspace) with $\tilde{V}_0 = E \oplus E^*$ and $\tilde{V}_1 = V$. Define a nondegenerate symmetric bilinear form on $E \oplus E^*$ by $\langle (e, \phi), (e', \phi') \rangle = \phi'(e) + \phi(e')$. Then $\tilde{V}$ has a supersymmetric bilinear form by taking the direct sum of the orthogonal form on $E \oplus E^*$ and the symplectic form on $V$. Let $\osp(\tilde{V})$ be the orthosymplectic Lie superalgebra in $\fgl(\tilde{V})$ compatible with this form. See \cite[\S 1.1.3]{chengwang} for more details on supersymmetric forms and orthosymplectic Lie superalgebras. We have a $\bZ$-grading on $\osp(\tilde{V})$ supported on $[-2,2]$ (we use the symbol $\oplus$ to separate the different pieces of this grading):
\[
\lw^2(E) \oplus (E \otimes V) \oplus (\fgl(E) \times \fsp(V)) \oplus (E^* \otimes V^*) \oplus \lw^2(E^*).
\]
Let $\fg$ be the positive part of this grading, i.e.,
\[
\fg = (E^* \otimes V^*) \oplus \lw^2(E^*).
\]
Then $\lw^2(E^*)$ is central in $\fg$. For $e \otimes v, e' \otimes v' \in E^* \otimes V^*$, we have $[e \otimes v, e' \otimes v'] = \omega_V(v,v') e \wedge e'$ where $\omega_V$ is the symplectic form on $V^* \cong V$. So Proposition~\ref{prop:CI-koszul} implies the following:

\begin{proposition}  
The Koszul dual of $B$ is the universal enveloping algebra $\rU(\fg)$.
\end{proposition}

Let $F \subset V$ be a maximal isotropic subspace and write $V = F \oplus F^*$. Let $\ol{\fu} = \lw^2(E^*|F) \subset \osp(\tilde{V})$. It inherits a $\bZ$-grading: 
\[
\ol{\fu}_0 = \Sym^2(F),  \quad \ol{\fu}_1 = E^* \otimes F,  \quad \ol{\fu}_2 = \lw^2(E^*).
\]
Let $\fg' \subset \osp(\tilde{V})$ be the subalgebra generated by $\ol{\fu}$ and $\fg$. Then it also has a $\bZ$-grading:
\[
\fg'_0 = \Sym^2(F), \quad \fg'_1 = E^* \otimes V^*,  \quad \fg'_2 = \lw^2(E^*).
\]

\begin{lemma} \label{lem:lie-exact}
We have the following two exact sequences of Lie superalgebras
\begin{subeqns}
\begin{align}
0 \to \fg \to \fg' \to \Sym^2(F) \to 0, \label{eqn:ext-1}\\
0 \to \ol{\fu} \to \fg' \to E^* \otimes F^* \to 0 \label{eqn:ext-2}.
\end{align}
\end{subeqns}
\end{lemma}

\begin{proof}
For \eqref{eqn:ext-1}, it follows from the $\bZ$-grading that $\fg$ is an ideal in $\fg'$. The same reasoning almost works for \eqref{eqn:ext-2}, but we also have to note that $[\Sym^2(F), E^* \otimes F^*] = E^* \otimes F$.
\end{proof}

\subsection{Howe duality} \label{sec:howe}

In this section, we assume that $\ch(\bk)=0$.

Let $U$ be a $2k$-dimensional symplectic space. Then $U \otimes \tilde{V}$ is a superspace with even part $U \otimes (E \oplus E^*)$ and odd part $U \otimes V$ and has a super skew-symmetric bilinear form by taking the tensor product of the forms on $U$ and $\tilde{V}$. Let $\spo(U \otimes \tilde{V})$ be the associated orthosymplectic Lie superalgebra (we use $\spo$ instead of $\osp$ since our form is super skew-symmetric).

Let $E^*|F$ denote the superspace with even part $E^*$ and odd part $F$. Then $U \otimes (E^*|F)$ is a maximal isotropic subspace of $U \otimes \tilde{V}$ and hence we get an oscillator representation
\[
\sU = \Sym(U \otimes (E^*|F))
\]
of $\spo(U \otimes \tilde{V})$. Both $\fsp(U)$ and $\osp(\tilde{V})$ are subalgebras of $\spo(U \otimes \tilde{V})$ which commute with one another, so $\fsp(U) \times \osp(\tilde{V})$ acts on $\sU$. As a representation of $\fsp(U)$, $\sU$ is a direct sum of finite-dimensional representations, so we can use the group $\Sp(U)$ instead. 

For a partition $\lambda$ with $\ell(\lambda) \le k$ and $\lambda_{n+1} \le n+d$, let $\cN^k_\lambda$ denote the $\bS_{[\lambda]}(U)$-isotypic component of $\sU$ under the action of $\Sp(U)$. We need some of the following facts about the action of $\Sp(U) \times \osp(\tilde{V})$, which state that they form a Howe dual pair (we will not use unitarizability, and hence do not define it; we state it only for completeness):

\begin{theorem}
For a partition $\lambda$ with $\ell(\lambda) \le k$ and $\lambda_{n+1} \le n+d$, $\cN^k_\lambda$ is an irreducible, unitarizable, lowest-weight representation of $\osp(\tilde{V})$. We have $\cN^k_\lambda \cong \cN^k_{\lambda'}$ if and only if $\lambda = \lambda'$. As a representation of $\Sp(U) \times \osp(\tilde{V})$, we have a direct sum decomposition 
\[
\sU = \Sym(U \otimes (E^*|F)) = \bigoplus_{\substack{\ell(\lambda) \le k\\ \lambda_{n+1} \le n+d}} \bS_{[\lambda]}(U) \otimes \cN^k_\lambda.
\]
\end{theorem}

\begin{proof} 
See \cite[\S\S 5.3.1, 5.3.2]{chengwang}. 
\end{proof}

\begin{theorem} \label{thm:super-lwood}
We have a $\fgl(E) \times \fgl(F)$-equivariant isomorphism
\[
\rH_i(\ol{\fu}; \cN^k_\lambda) = \Tor_i^{\rU(\ol{\fu})}(\cN^k_\lambda, \bk) \cong \bigoplus_{\substack{\alpha\\ \tau_{2k}(\alpha) = \lambda \\ \iota_{2k}(\alpha) = i}} \bS_\alpha(E^*|F) \otimes (\det E^*)^k \otimes (\det F^*)^k.
\]
\end{theorem}

\begin{proof}
This follows from \cite[Theorem 5.7]{CKW}. Their notation does not match ours, so alternatively one can use \cite[\S 3]{lwood} by noting that \cite[Corollary 3.16]{lwood} is valid whenever $E$ is an object in a semisimple monoidal Abelian category. In particular, we can take $E$ to be $E^*|F$ in the category of polynomial representations of $\fgl(E^*|F)$. We note that the twist by $(\det E^*)^k \otimes (\det F^*)^k$ is not present in \cite{lwood} because it did not consider the action of $\ol{\fu}$ as being restricted from a larger Lie (super)algebra which forces the twist present here.
\end{proof}

\begin{proposition} \label{prop:typeC-linear-pres}
Let $\lambda$ be a partition with $\ell(\lambda) \le k$ and $\lambda_{n+1} \le n+d$. Set $\nu = (\lambda_1^\dagger, \dots, \lambda_{n+d}^\dagger)$ and $\mu = (\lambda_{n+d+1}^\dagger, \lambda_{n+d+2}^\dagger, \dots)$. Then
\begin{align*}
\rH_0(\fg; \cN^k_\lambda) &= (\det E^*)^k \otimes \bS_{\mu^\dagger}(E^*) \otimes \bS_{[((n+d) \times k) \setminus \nu]}(V) \\
\rH_1(\fg; \cN^k_\lambda) &= (\det E^*)^k \otimes \bS_{(1 + \nu_{n+d}, \mu_2, \mu_3, \dots)^\dagger}(E^*) \otimes \bS_{[(k-\mu_1 + 1, k - \nu_{n+d-1}, \dots, k - \nu_1)]}(V).
\end{align*}
In particular, as a $\rU(\fg)$-module, $\cN^k_\lambda$ is generated in a single degree and has relations of degree $\lambda_{n+d}^\dagger - \lambda_{n+d+1}^\dagger + 1$. 
\end{proposition}

\begin{subeqns}
\begin{proof}
Consider $\cN^k_\lambda$ as a $\fg'$-module. We will analyze the Hochschild--Serre spectral sequence \eqref{eqn:HS-SS} with respect to the exact sequences \eqref{eqn:ext-1} and \eqref{eqn:ext-2}. Everything is equivariant with respect to $\fgl(E) \times \fgl(F)$ so we will keep track of this symmetry.

For $\rH_0$, the Hochschild--Serre spectral sequence gives us the following identity:
\begin{align*}
\rH_0(\Sym^2(F); \rH_0(\fg; \cN^k_\lambda)) \cong \rH_0(\fg'; \cN^k_\lambda) \cong \rH_0(E^* \otimes F^*; \rH_0(\ol{\fu}; \cN^k_\lambda)).
\end{align*}
From Theorem~\ref{thm:super-lwood}, 
\begin{align} \label{eqn:Nlambda-H0u}
\rH_0(\ol{\fu}; \cN^k_\lambda) = \bS_\lambda(E^*|F) \otimes (\det E^*)^k \otimes (\det F^*)^k.
\end{align}
Proposition~\ref{prop:schur-hom} gives
\[
\rH_0(E^* \otimes F^*; \bS_\lambda(E^*|F)) = \bS_{\nu}(F) \otimes \bS_{\mu^\dagger}(E^*).
\]
Hence
\[
\rH_0(\Sym^2(F); \rH_0(\fg; \cN^k_\lambda)) = \bS_{(k^n) + \mu^\dagger}(E^*) \otimes \bS_{((n+d) \times k) \setminus \nu}(F^*),
\]
so by Proposition~\ref{prop:par-hom}, we have 
\begin{align} \label{eqn:Nlambda-H0}
\rH_0(\fg; \cN^k_\lambda) = \bS_{(k^n) + \mu^\dagger}(E^*) \otimes \bS_{[((n+d) \times k) \setminus \nu]}(V).
\end{align}

Now we calculate $\rH_1$. First consider the Hochschild--Serre spectral sequence for \eqref{eqn:ext-1}. As a $\fgl(E) \times \fgl(F)$-representation, $\rH_1(\fg'; \cN^k_\lambda)$ is the direct sum of $\rE^2_{1,0}$ and the cokernel of $\rd^2 \colon \rE^2_{2,0} \to \rE^2_{0,1}$. We claim that $\rd^2 = 0$. First, $\rE^2_{2,0} = \rH_2(\Sym^2(F); \rH_0(\fg; \cN^k_\lambda))$ and $\rE^2_{0,1} = \rH_0(\Sym^2(F); \rH_1(\fg; \cN^k_\lambda))$. The action of the Lie algebra $\Sym^2(F)$ commutes with the action of $\fgl(E)$. From \eqref{eqn:Nlambda-H0}, $\rH_0(\fg; \cN^k_\lambda)$ is $\fgl(E)$-isotypic, i.e., is a direct sum of copies of the same representation $\bS_{(k^n)+\mu^\dagger}(E^*)$. Since $\rH_1(\fg; \cN^k_\lambda) \subset \rH_0(\fg; \cN^k_\lambda) \otimes \rU(\fg)_{>0}$, and $\rU(\fg)_{>0}$ does not contain any $\fgl(E)$-invariants, we conclude that the $\bS_{(k^n)+\mu^\dagger}(E^*)$-isotypic component of $\rH_1(\fg; \cN^k_\lambda)$ is $0$. Hence $\rd^2 = 0$. So we get
\begin{align*}
\rH_1(\fg'; \cN^k_\lambda) \cong \rH_0(\Sym^2(F); \rH_1(\fg; \cN^k_\lambda)) \oplus \rH_1(\Sym^2(F); \rH_0(\fg; \cN^k_\lambda)).
\end{align*}
So, to calculate $\rH_1(\fg; \cN^k_\lambda)$, we first remove the $\bS_{(k^n)+\mu^\dagger}(E^*)$-isotypic component from $\rH_1(\fg'; \cN^k_\lambda)$.

Now consider the Hochschild--Serre spectral sequence of \eqref{eqn:ext-2}. By Theorem~\ref{thm:super-lwood}, 
\[
\rH_1(\ol{\fu}; \cN^k_\lambda) = \bS_{(\lambda, 1^{2k+2-2\ell(\lambda)})}(E^*|F) \otimes (\det E^*)^k \otimes (\det F^*)^k.
\]
Setting $\nu' = (2k+2-\ell(\lambda), \nu_2, \dots, \nu_{n+d})$ (the first $n+d$ columns of $(\lambda, 1^{2k+2-2\ell(\lambda)})$), Proposition~\ref{prop:schur-hom} gives
\[
\rE^2_{0,1} = \rH_0(E^* \otimes F^*; \rH_1(\ol{\fu}; \cN^k_\lambda)) = \bS_{(k^n)+\mu^\dagger}(E^*) \otimes (\det F^*)^k \otimes \bS_{\nu'}(F),
\]
so we can ignore it for the purposes of calculating $\rH_1(\fg; \cN^k_\lambda)$, i.e., 
\begin{align*}
\rH_0(\Sym^2(F); \rH_1(\fg; \cN^k_\lambda)) &\cong \rE^2_{1,0} = \rH_1(E^* \otimes F^*; \rH_0(\ol{\fu}; \cN^k_\lambda))\\
& = (\det F^*)^k \otimes \bS_{(\nu_1, \dots, \nu_{n+d-1}, \mu_1 -1)}(F) \otimes (\det E^*)^k \otimes \bS_{(1 + \nu_{n+d}, \mu_2, \mu_3, \dots)^\dagger}(E^*).
\end{align*}
One more application of Proposition~\ref{prop:par-hom} gives us the desired formula for $\rH_1(\fg; \cN^k_\lambda)$.
\end{proof}
\end{subeqns}

\begin{proposition} \label{prop:ss-chase}
If $\lambda_1 \le d$ and $\ell(\lambda) \le k$, then $\cN^k_\lambda$ has a linear free resolution over~$\rU(\fg)$.
\end{proposition}

\begin{proof}
Consider $\cN^k_\lambda$ as a $\fg'$-module. We will analyze the Hochschild--Serre spectral sequence \eqref{eqn:HS-SS} with respect to the exact sequences \eqref{eqn:ext-1} and \eqref{eqn:ext-2}. Everything is equivariant with respect to $\fgl(E) \times \fgl(F)$ so we will keep track of this symmetry. In particular, the center of $\fgl(E)$ puts a grading on all vector spaces that are involved according to which multiple of the trace character it acts by (we use the convention that $D$ times the trace means degree $-D$). We will prove by induction on $i$ that $\rH_i(\fg; \cN^k_\lambda)$ is concentrated in degree $i+nk$. The base case $i=0$ follows from Proposition~\ref{prop:typeC-linear-pres}. Note that $\rH_i(\fg; \cN^k_\lambda)$ is concentrated in degrees $\ge i+nk$ since the smallest degree appearing in $\rH_i$ is strictly bigger than the smallest degree appearing in $\rH_{i-1}$ (use that $\rH_i \subset \rH_{i-1} \otimes \rU(\fg)_{>0}$).

Since the action of $\Sym^2(F)$ commutes with the $\fgl(E)$-action, this is the same as proving that $\rH_0(\Sym^2(F); \rH_i(\fg; \cN^k_\lambda))$ is concentrated in degree $i+nk$. In the Hochschild--Serre spectral sequence for \eqref{eqn:ext-1}, this is the term $\rE^2_{0,i}$. We claim that it is the same as $\rE^\infty_{0,i}$. The relevant differentials are of the form $\rE^r_{r,i-r+1} \to \rE^r_{0,i}$. But $\rE^r_{r,i-r+1}$ is a subquotient of $\rE^2_{r,i-r+1} = \rH_r(\Sym^2(F); \rH_{i-r+1}(\fg; \cN^k_\lambda))$. By induction, the latter is concentrated in degree $i-r+1+nk$, and hence the differential must be $0$, and the claim is proven. Now, we have $\rH_i(\fg'; \cN^k_\lambda) = \bigoplus_{j=0}^i \rE^\infty_{j,i-j}$. For $j>0$, we know that $\rE^\infty_{j,i-j}$, being a subquotient of $\rH_j(\Sym^2(F); \rH_{i-j}(\fg; \cN^k_\lambda))$, is concentrated in degree $i-j+nk$ by induction. In conclusion, $\rH_0(\Sym^2(F); \rH_i(\fg; \cN^k_\lambda))$ is the degree $\ge i+nk$ part of $\rH_i(\fg'; \cN^k_\lambda)$.

Using the Hochschild--Serre spectral sequence for \eqref{eqn:ext-2}, $\rH_i(\fg'; \cN^k_\lambda)$ is a subquotient of $\bigoplus_{j=0}^i \rH_j(E^* \otimes F^*; \rH_{i-j}(\ol{\fu}; \cN^k_\lambda))$. We will show that this space is concentrated in degree $\le i+nk$, which implies the result. Pick a partition $\alpha$ such that $\bS_\alpha(E^*|F) \otimes (\det E^*)^k \otimes (\det F^*)^k \subseteq \rH_m(\ol{\fu}; \cN^k_\lambda)$. Then by Theorem~\ref{thm:super-lwood}, we have $\tau_{2k}(\alpha) = \lambda$ and $\iota_{2k}(\alpha) = m$. Since $\bS_\alpha(E^*|F) \ne 0$, we have $\alpha_{n+1} \le n+d$. Let $\rho = (-1,-2,-3,\dots)$. By Proposition~\ref{prop:schur-hom}, to calculate $\rH_{m'}(E^* \otimes F^*; \bS_\alpha(E^*|F))$, we apply the $\rho$-shifted action of a length $m'$ permutation $w$ to $\alpha^\dagger$ to get $(\beta_1, \dots, \beta_{n+d}, \gamma_1, \gamma_2, \dots)$ so that $\beta$ and $\gamma$ are weakly decreasing and $\gamma$ is nonnegative. This contributes the representation $\bS_\beta(F) \otimes \bS_{\gamma^\dagger}(E^*)$ (and so $\gamma_1 \le n$). So we just have to show that $|\gamma| \le m + m'$. We isolate and prove this statement in Lemma~\ref{lem:gamma-bound} along with a calculation of what happens when $|\gamma|=m+m'$.
\end{proof}

\begin{lemma} \label{lem:gamma-bound}
Fix $k > 0$. Let $\lambda$ be a partition with $\lambda_1 \le d$ and $\ell(\lambda) \le k$. Pick a partition $\alpha$ with $\alpha_{n+1} \le n+d$ such that $\tau_{2k}(\alpha) = \lambda$ and let $w$ be a permutation such that $w \bullet \alpha^\dagger = (\beta_1, \dots, \beta_{n+d}, \gamma_1, \gamma_2, \dots)$ with $\beta_1 \ge \cdots \ge \beta_{n+d}$ and $n \ge \gamma_1 \ge \gamma_2 \ge \cdots \ge 0$. 

Then $|\gamma| \le \iota_{2k}(\alpha) + \ell(w)$. If $|\gamma| = \iota_{2k}(\alpha) + \ell(w)$, then:
\begin{itemize}[\indent -]
\item $\iota_{2k}(\alpha) = 0$ and $\alpha^\dagger = \lambda^\dagger = (\beta_1, \dots, \beta_d)$, and 
\item $w \bullet \alpha^\dagger = (\lambda_1^\dagger, \dots, \lambda_d^\dagger, -\gamma^\dagger_n, \dots, -\gamma^\dagger_1, \gamma_1, \gamma_2, \dots)$ where $|\gamma|=\ell(w)$.
\end{itemize}
\end{lemma}

\begin{subeqns}
\begin{proof}
We claim that
\begin{align} \label{eqn:iota-bound1}
\iota_{2k}(\alpha) \ge \sum_{j \ge 1} \max(\alpha_j - d - j + 1, 0).
\end{align}
The modification rule (\S\ref{sec:typeC-weyl}) expresses $\alpha \setminus \lambda$ as a disjoint union of border strips. Since $\lambda_1 \le d$, the part which is to the right of the $d$th column of $\alpha$  is expressed as a disjoint union of border strips, and the right hand side of \eqref{eqn:iota-bound1} is the least number of columns needed in all of the border strips to do this. This proves the claim.

Set $\nu = (\alpha_{n+d+1}^\dagger, \alpha_{n+d+2}^\dagger, \dots)$. If $w = 1$ is the identity, then $|\gamma| = |\nu| = \sum_{j \ge 1} \max(\alpha_j - n - d, 0)$. Since $\alpha_{n+1} \le n+d$, we conclude that this sum is $\le \iota_{2k}(\alpha)$.

Otherwise, suppose $w \ne 1$. To clarify the remainder of this argument, we use the convention that $w(a_1, a_2, \dots) = (a_{w^{-1}(1)}, a_{w^{-1}(2)}, \dots)$ and set $v = w^{-1}$. Since $\beta$ and $\gamma$ are weakly decreasing, we get $v(1) < v(2) < \cdots < v(n+d)$ and $v(n+d+1) < v(n+d+2) < \cdots$. So the number of inversions of $w$ is $\ell(w) = \ell(v) =  \sum_{j=1}^{n+d} (v(j) - j)$, and hence
\begin{equation} \label{eqn:gamma-len}
\begin{split}
|\gamma| = |\alpha| - |\beta| = \sum_{j \ge 1} (\alpha_j^\dagger - \beta_j) &= \sum_{j > n+d} \alpha_j^\dagger + \sum_{j = 1}^{n+d} (\alpha_j^\dagger - \alpha_{v(j)}^\dagger + v(j) - j)\\
&= \sum_{j > n+d} \alpha_j^\dagger + \sum_{j = 1}^{n+d} (\alpha_j^\dagger - \alpha_{v(j)}^\dagger) + \ell(w).
\end{split}
\end{equation}
Here we use the notation $|\beta|=\sum_{j=1}^{n+d} \beta_j$ even though some entries of $\beta$ may be negative.

Let $i$ be minimal with the property that $v(i) \ne i$. From above, we get $i \le n+d$ and $i = v(n+d+1)$. Then $\gamma_1 = \alpha_i^\dagger + n + d + 1 - i$. Since $\gamma_1 \le n$, we get $0 \ge \alpha_i^\dagger + d + 1 - i$, i.e., $\alpha_i^\dagger < i-d$, and hence $\alpha_{i-d} < i$. So 
\begin{align} \label{eqn:iota-bound2}
\sum_{j \ge i} \alpha_j^\dagger = \sum_{j=1}^{i-d-1} \max(\alpha_j - i + 1, 0) \le \sum_{j=1}^{i-d-1} \max(\alpha_j - d - j, 0) \le \iota_{2k}(\alpha)
\end{align}
where the last inequality is \eqref{eqn:iota-bound1}. By definition of $i$, and nonnegativity of $\alpha$, we have $\sum_{j \ge i} \alpha_j^\dagger \ge \sum_{j>n+d} \alpha_j^\dagger + \sum_{j=1}^{n+d} (\alpha_j^\dagger - \alpha_{v(j)}^\dagger)$, so combining this with \eqref{eqn:gamma-len}, we conclude that $|\gamma| \le \iota_{2k}(\alpha) + \ell(w)$.

Now assume that $|\gamma| = \iota_{2k}(\alpha) + \ell(w)$. In particular, all of the above inequalities are equalities. From our proof of \eqref{eqn:iota-bound1}, it can only be an equality if $d=0$ or $\iota_{2k}(\alpha)=0$ (if $d>0$ and $\iota_{2k}(\alpha)>0$, then the boxes in the first $d$ columns of $\alpha$ contribute to $\iota_{2k}(\alpha)$). Suppose further that $d=0$. By \eqref{eqn:iota-bound2}, we have $\max(\alpha_j - i + 1, 0) = \max(\alpha_j - j, 0)$ for $j=1,\dots,i-1$. If $i=1$, then the sum in \eqref{eqn:iota-bound2} is empty, which means $\iota_{2k}(\alpha)=0$. Otherwise $i>1$, so $\max(\alpha_1 - i + 1,0) = \max(\alpha_1 - 1,0)$, i.e., $\alpha_1 \le 1$. This implies that the sum in \eqref{eqn:iota-bound2} is empty, so again we get $\iota_{2k}(\alpha)=0$. In conclusion, we have $\iota_{2k}(\alpha) = 0$ if $|\gamma| = \iota_{2k}(\alpha) + \ell(w)$ and hence that $\alpha^\dagger = \lambda^\dagger = (\beta_1, \dots, \beta_d)$ (by assumption, $\lambda_1 \le d$).

So we are in the situation where $w$ acts on a sequence of zeroes. The result of this is of the form $(-\gamma_n^\dagger, \dots, -\gamma_1^\dagger, \gamma_1, \gamma_2, \dots)$ where $|\gamma| = \ell(w)$ (see, for example, \cite[\S 2.4]{getzler}), which proves the last part of the lemma.
\end{proof}
\end{subeqns}

\begin{corollary} \label{cor:Nlambda-upper}
Let $\lambda$ be a partition with $\lambda_1 \le d$ and $\ell(\lambda) \le k$. As a representation of $\GL(E) \times \Sp(V)$, $\rH_i(\fg; \cN_\lambda^k)$ is a subrepresentation of 
\[
\bigoplus_{\substack{|\gamma| = i\\ \ell(\gamma) \le n}} \bS_{(k^n) + \gamma}(E^*) \otimes \bS_{[((k^n) + \gamma, k - \lambda^\dagger_d, \dots, k - \lambda^\dagger_1)]}(V).
\]
\end{corollary}

\begin{proof}
In the proof of Proposition~\ref{prop:ss-chase}, we defined the degree of a representation in terms of the $\fgl(E)$ action and showed that $\rH_0(\Sym^2(F); \rH_i(\fg; \cN_\lambda^k))$ is the degree $i+nk$ part of $\rH_i(\fg'; \cN_\lambda^k)$, which is a subquotient (by semisimplicity, we can say subrepresentation) of the degree $i+nk$ part of $\bigoplus_{j=0}^i \rH_j(E^* \otimes F^*; \rH_{i-j}(\ol{\fu}; \cN_\lambda^k))$. Combining the calculation of this space in the proof of Proposition~\ref{prop:ss-chase} with Lemma~\ref{lem:gamma-bound} gives that the latter representation is 
\[
P := \bigoplus_{\substack{|\gamma| = i\\ \ell(\gamma) \le n}} \bS_{(k^n) + \gamma}(E^*) \otimes \bS_{((k^n) + \gamma,k - \lambda_d^\dagger, \dots, k - \lambda_1^\dagger)}(F^*).
\]
If $\rH_0(\Sym^2(F); P_0) = P$, then $\rH_i(\fg; \cN_\lambda^k) \subseteq P_0$. Using Proposition~\ref{prop:par-hom}, we get 
\[
P_0 = \bigoplus_{\substack{|\gamma|=i\\ \ell(\gamma) \le n}} \bS_{(k^n) + \gamma}(E^*) \otimes \bS_{[((k^n) + \gamma, k - \lambda_d^\dagger, \dots, k - \lambda_1^\dagger)]}(V). \qedhere
\]
\end{proof}

\subsection{Main result and consequences} \label{sec:main-result-typeC}

We assume that $\ch(\bk)=0$ in this section.

Recall the $\bZ$-grading on $\fosp(\tilde{V})$ from \S\ref{sec:typeC-prelim}. Let $\fp$ be the sum of the nonpositive pieces. So 
\[
\fp = \lw^2(E) \oplus (E \otimes V) \oplus (\fgl(E) \times \fsp(V)).
\]
Given a representation $W$ of $\fgl(E) \times \fsp(V)$, extend it to a $\fp$-module by letting the negative part act by $0$. The induced representation is ${\rm Ind}_\fp^{\fosp(\tilde{V})}(W) = \rU(\fosp(\tilde{V})) \otimes_{\rU(\fp)} W$, which is a representation of $\fosp(\tilde{V})$. When $W$ is irreducible, we call this a {\bf parabolic Verma module}. As a $\rU(\fg)$-module, it is isomorphic to $\rU(\fg) \otimes W$. From the definition, if $\cW$ is a representation of $\fosp(\tilde{V})$, and ${\rm Ind}_\fp^{\fosp(\tilde{V})}(W) \to \cW$ is $\rU(\fg)$-equivariant, then it is  $\fosp(\tilde{V})$-equivariant if and only if the negative part of $\fosp(\tilde{V})$ acts as $0$ on the image of $W$ in $\cW$. This implies:

\begin{lemma} \label{lem:parabolic-verma}
Let $\cW_0$ be a graded $\fosp(\tilde{V})$-submodule of ${\rm Ind}_\fp^{\fosp(\tilde{V})}(W)$ which is generated by $W_0$ as a $\rU(\fg)$-module, so we have a surjection $\rU(\fg) \otimes W_0 \to \cW_0 \subset {\rm Ind}_\fp^{\fosp(\tilde{V})}(W)$. If $W_0$ is concentrated in a single degree, then this map is equivariant for the natural action of $\fosp(\tilde{V})$.
\end{lemma}

\begin{theorem} \label{thm:main-typeC}
Let $\lambda$ be a partition with $\ell(\lambda) \le k$ and $\lambda_{n+1} \le n+d$. 

\noindent Set $\nu = (\lambda_1^\dagger, \dots, \lambda_{n+d}^\dagger)$ and $\mu = (\lambda_{n+d+1}^\dagger, \lambda_{n+d+2}^\dagger, \dots)$. Then:
\begin{compactenum}[\rm (a)]
\item $\lambda^\dagger_{n+d} = \lambda^\dagger_{n+d+1}$ if and only if $\cN^k_\lambda$ has a linear presentation as a $\rU(\fg)$-module. 
\item If $\lambda^\dagger_{n+d} = \lambda^\dagger_{n+d+1}$, $\bL(\cN^k_\lambda)$ (notation of \S\ref{sec:koszul}) is the first linear strand of a $B$-module $M$ with a compatible $\GL(E) \times \Sp(V)$-action and which is uniquely determined by the fact that it is generated in degree $0$ with degree $1$ relations, and satisfies
\begin{align*}
M_0 &= (\det E)^k \otimes \bS_{\mu^\dagger}(E) \otimes \bS_{[((n+d) \times k) \setminus \nu]}(V),\\
M_1 &= (\det E)^k \otimes \bS_{(\mu_1 + 1, \mu_2, \mu_3, \dots)^\dagger}(E) \otimes \bS_{[(k+1-\nu_{n+d}, k-\nu_{n+d-1}, \dots, k-\nu_1)]}(V).
\end{align*}
\end{compactenum}
If furthermore $\lambda_1 \le d$, then $M \cong \cM^k_{\lambda^\dagger}$, and:
\begin{compactenum}[\rm (a)]
\setcounter{enumi}{2}
\item $\bL(\cN^k_{\lambda})$ is an acyclic linear complex and is a minimal free resolution of $\cM^k_{\lambda^\dagger}$ over $B$,
\item $\bR(\cM^k_{\lambda^\dagger})$ is an acyclic linear complex and is a minimal free resolution of $\cN^k_\lambda$ over $\rU(\fg)$. The action of $\fg$ on this complex extends to an action of $\osp(\tilde{V})$ and each $\bR(\cM^k_{\lambda^\dagger})_i$ is a direct sum of parabolic Verma modules.
\end{compactenum}
\end{theorem}

\begin{proof}
(a) follows directly from Proposition~\ref{prop:typeC-linear-pres}. By Proposition~\ref{prop:typeC-linear-pres} and Theorem~\ref{thm:dict}, $\bL(\cN^k_\lambda)$ is the first linear strand of the cokernel of $(\cN^k_\lambda)_1^* \otimes B(-1) \to (\cN^k_\lambda)_0^* \otimes B$, and $(\cN^k_\lambda)_0^* = (\det E)^k \otimes \bS_{\mu^\dagger}(E) \otimes \bS_{[((n + d) \times k) \setminus \nu]}(V)$. Its tensor product with $B_1 = E \otimes V$ is multiplicity-free: for all classical groups, the tensor product of an irreducible representation with the vector representation is multiplicity-free \cite[Theorem 1.1ABCD]{stembridge} (there the character of the vector representation is denoted $\chi(\omega_n)$). So the cokernel $M$ is uniquely determined by the condition that it is linearly presented as a $B$-module and $M_i = \rH_i(\fg; \cN^k_\lambda)$ for $i=0,1$. This proves (b).

When $\lambda_1 \le d$, the module $\cM^k_{\lambda^\dagger}$ satisfies these conditions except that we do not yet know if it has a linear presentation. We have a surjection $M \to \cM^k_{\lambda^\dagger}$, which is an isomorphism by comparing Corollary~\ref{cor:Nlambda-upper} and \eqref{eqn:cM-module}. So (c) and (d) are an immediate consequence of Proposition~\ref{prop:ss-chase} and Theorem~\ref{thm:dict}. The second statement of (d) follows from Lemma~\ref{lem:parabolic-verma} by induction on $i$ since the free resolution of $\cN_\lambda^k$ over $\rU(\fg)$ is linear.
\end{proof}

\begin{conjecture} \label{conj:Ilambda}
Let $I_\lambda$ be the ideal generated by $\bS_\lambda(E) \otimes \bS_{[\lambda]}(V)$ in $B$. Then the resolution of $I_\lambda$ is a representation of $\osp(\tilde{V})$. Furthermore, each linear strand is irreducible.
\end{conjecture}

Now we derive a few consequences of Theorem~\ref{thm:main-typeC}.

First, we begin with an alternative construction of $\cM^k_\nu$. Consider $\IGr(n,V)$ with vector bundles $\eta = E \otimes \cR^*$, and $\cV^k_\nu = (\det E)^k \otimes (\det \cR^*)^k \otimes \bS_{[(d \times k) \setminus \nu]}(\cR^\perp / \cR)$. 

\begin{proposition} \label{prop:Mnu-alt}
$\cM^k_\nu = \rH^0(\IGr(n,V); \Sym(\eta) \otimes \cV^k_\nu)$.
\end{proposition}

\begin{proof}
Let $X$ be the partial isotropic flag variety ${\bf IFl}(n,n+d,V)$ with tautological flag $\cR_n \subset \cR_{n+d} \subset V$. Let $\pi_1 \colon X \to \IGr(n,V)$ and $\pi_2 \colon X \to \IGr(n+d,V)$ be the projections. Let $\cR$ and $\cR'$ be the tautological subbundles on $\IGr(n,V)$ and $\IGr(n+d,V)$, respectively.

First, $\pi_1$ is the relative Grassmannian $\Gr(d,\cR^\perp / \cR)$. So 
\[
\bS_{[(d \times k) \setminus \nu]}(\cR^\perp / \cR) = {\pi_1}_* (\bS_{(d \times k) \setminus \nu}(\cR_n^\perp / \cR_{n+d}))
\]
using Theorem~\ref{thm:bott-C} and Remark~\ref{rmk:bw-functorial}. Set 
\[
\cW = \Sym(E \otimes \cR_n^*) \otimes (\det E)^k \otimes (\det \cR_n^*)^k \otimes \bS_{(d \times k) \setminus \nu}(\cR_n^\perp / \cR_{n+d}).
\]
Since $\cR_n = \pi_1^*(\cR)$, the projection formula gives $\Sym(\eta) \otimes \cV_\nu^k = {\pi_1}_*(\cW)$. Next, $\pi_2$ is the relative Grassmannian $\Gr(n,\cR')$. The symplectic form on $V$ induces an isomorphism $\cR_n^\perp / \cR_{n+d} \cong \cR_{n+d}^* / \cR_n^*$. So applying ${\pi_2}_*$ to $\cW$ is a relative version of the construction in \S\ref{sec:EN-powers}, which is how we defined $\sM^k_\nu(E, {\cR'}^*)$.

So pushing forward $\Sym(\eta) \otimes \cV_\nu^k$ to a point is the same as pushing forward $\sM^k_\nu(E, {\cR'}^*)$ to a point. Since the latter is $\cM^k_\nu$, we are done.
\end{proof}

Define the following vector bundle on $\IGr(n,V)$:
\[
\cE^\nu_\lambda = \bS_\lambda(\cR^\perp) \otimes (\det \cR^*)^k \otimes \bS_{[(d \times k) \setminus \nu]}(\cR^\perp / \cR).
\]

\begin{corollary} \label{cor:Mnu-res-general}
Let $\bF^\nu_\bullet$ be the minimal free resolution of $\cM^k_\nu$ over $A$. Then
\begin{align*} 
\bF^\nu_i = \bigoplus_{j \ge 0} \rH^j(\IGr(n,V); \bigoplus_{\substack{|\lambda| = i+j\\ \lambda \subseteq n \times (n+2d)}} \bS_{(k^n) + \lambda}(E) \otimes \cE^\nu_{\lambda^\dagger}) \otimes A(-i-j).
\end{align*}
\end{corollary}

\begin{proof}
In the notation of \S\ref{sec:geom}, $\xi = E \otimes \cR^\perp$. Now use
Proposition~\ref{prop:Mnu-alt}, Theorem~\ref{thm:geom-tech}, and Theorem~\ref{thm:cauchy-id}.
\end{proof}

It is probably difficult to give a closed formula to simplify the above expression, but the next theorem sheds some light on the relevant combinatorial features. While this result looks similar to the Borel--Weil--Bott theorem, it is not a special case because the vector bundles $\cE^\nu_\lambda$ are not irreducible homogeneous bundles. To strengthen the analogy, we point the reader to \cite[\S 3.5]{lwood} where $\iota_{2k}$ and $\tau_{2k}$ are defined in terms of the Weyl group of type $\rB\rC_\infty$.

\begin{theorem} \label{thm:bott-gen}
Let $\nu \subseteq (k^d)$ and $\lambda \subseteq n \times (n+2d)$ be partitions and let $\cR$ be the tautological subbundle on $\IGr(n,V)$. 
\begin{compactenum}[\rm (a)]
\item If $k \ge \ell(\lambda)$, then $\rH^i(\IGr(n,V); \cE^\nu_{\lambda^\dagger}) = 0$ for $i>0$.
\item The cohomology of $\cE^\nu_{\lambda^\dagger}$ vanishes unless $\tau_{2k}(\lambda) = \mu$ is defined, in which case the cohomology is nonzero only in degree $\iota_{2k}(\lambda) = i$, and we have an $\Sp(V)$-equivariant isomorphism
\begin{align*}
\rH^i(\IGr(n,V); \cE^\nu_{\lambda^\dagger}) \cong \rH^0(\IGr(n,V); \cE^\nu_{\mu^\dagger}).
\end{align*}
\end{compactenum}
\end{theorem}

\begin{proof}
Let $\bF_\bullet$ be the minimal free resolution of $\cM^k_\nu$ over $A$. Using \cite[\S 7]{eisenbud-ci}, we can use it to build a free resolution $\bG_\bullet$ of $\cM^k_\nu$ over $B$. We just need a few properties of this construction. First, the terms are given by
\[
\bG_i = \bigoplus_{j=0}^{\lfloor i/2\rfloor} \bF_{i-2j}(-2j) \otimes_\bk \rD^j(\lw^2(E)).
\]
Second, if we set $S = \Sym(\lw^2(E^*))$, then $\bR(\bG_\bullet)$ is a minimal complex of $S$-modules, and applying $\bL$ to its homology gives the minimal free resolution of $\cM^k_\nu$ over $B$. By Theorem~\ref{thm:main-typeC}, the minimal free resolution of $\cM^k_\nu$ is linear over $B$, which means that $\bR(\bG_\bullet)$ only has homology in degree $0$, and $\rH_0(\bR(\bG_\bullet)) = \cN^k_{\nu^\dagger}$. Putting this together implies
\[
\Tor^S_i(\cN^k_{\nu^\dagger}, \bk)_{j}^* = \Tor^A_{j}(\cM^k_\nu, \bk)_{i+j}.
\]
Now consider the $\GL(E) \times \Sp(V)$-action on $\Tor^S_i(\cN^k_{\nu^\dagger}, \bk)_j^*$. If $\bS_{(k^n)+\lambda}(E^*)$ has a nonzero $\GL(E)$-isotypic component, then $\tau_{2k}(\lambda)$ is well-defined and $\iota_{2k}(\lambda) = i$ by Theorem~\ref{thm:super-lwood} (or \cite[Proposition 3.13]{lwood}). In particular, since the $\GL(E)$ and $\Sp(V)$ actions commute, we have an isomorphism of $\Sp(V)$-representations
\[
\hom_{\GL(E)}(\bS_{(k^n)+\lambda}(E^*), \Tor^S_{\iota_{2k}(\lambda)}(\cN^k_{\nu^\dagger}, \bk)_{|\lambda|}^*) \cong \hom_{\GL(E)}(\bS_{(k^n)+\tau_{2k}(\lambda)}(E^*), \Tor^S_0(\cN^k_{\nu^\dagger}, \bk)_{|\tau_{2k}(\lambda)|}^*).
\]
Combining Theorem~\ref{thm:geom-tech} and Corollary~\ref{cor:Mnu-res-general}, we get
\[
\Tor^A_i(\cM^k_\nu, \bk)_{i+j} = \rH^j(\IGr(n,V); \bigoplus_{\substack{|\lambda| = i+j\\ \lambda \subseteq n \times (n+2d)}} \bS_{(k^n) + \lambda}(E) \otimes \cE^\nu_{\lambda^\dagger}).
\]
Combining everything, we get an isomorphism of $\Sp(V)$-representations
\begin{align*}
\rH^{\iota_{2k}(\lambda)}(\IGr(n,V); \cE^\nu_{\lambda^\dagger}) \cong \rH^0(\IGr(n,V); \cE^\nu_{\tau_{2k}(\lambda)^\dagger}),
\end{align*}
and that $\rH^N(\IGr(n,V); \cE^\nu_{\lambda^\dagger}) = 0$ if $N \ne \iota_{2k}(\lambda)$.
\end{proof}

The method of the proof of Theorem~\ref{thm:bott-gen} is a clean, elegant approach, but it is in some ways unsatisfactory since it appeals to so many other techniques. We have another proof which is more direct, but has the disadvantage of only working for $\nu = (k^d)$ and is much longer. Since it may be of independent interest, we include this alternate proof in \cite{lwood-minors-appendix}. 

\begin{remark} \label{rmk:typeC-minors-explain}
As in \S\ref{sec:typeC-prelim}, we interpret $A$ as the coordinate ring of the space of linear maps $E \to V^*$ and hence it makes sense to talk about the generic map with respect to $A$, i.e., picking bases for $E$ and $V^*$, the matrix entries are the variables of $A$. Similar comments apply to $B$. The ideal of maximal minors in $B$ of this generic matrix is generated by $\lw^n(E) \otimes \bS_{[1^n]}(V)$. When $\nu = (k^d)$, \eqref{eqn:cM-module} becomes 
\[
\cM^k_{(k^d)} = \bigoplus_{\lambda,\ \lambda_n \ge k} \bS_\lambda(E) \otimes \bS_{[\lambda]}(V),
\]
which is the $k$th power of the maximal minors in $B$ (use Proposition~\ref{prop:typeC-basicfacts}(c)).
\end{remark}

\begin{corollary} \label{cor:C-lin-res}
For $\nu \subseteq (k^d)$, if $k \ge n-1$, then the minimal free resolution of $\cM^k_\nu$ over $A$ is linear.
\end{corollary}

\begin{proof}
Combine Corollary~\ref{cor:Mnu-res-general} and Theorem~\ref{thm:bott-gen}(a).
\end{proof}

For every pair of finitely generated $B$-modules $M$ and $N$, there is an action of $\rU(\fg) \cong \ext^\bullet_B(\bk,\bk)$ on $\ext^\bullet_B(M,N)$ which we can restrict to $S = \Sym(\lw^2(E^*))$. The support of this module is the {\bf support variety} of the pair $(M,N)$, and is a subscheme of $\Spec(S) = \lw^2(E)$. The basic properties of support varieties are developed in \cite{support-ci}. We define the support variety of a single module $M$ to be the support variety of the pair $(M,\bk)$. Note that $\ext^i_B(M,\bk) = \Tor_i^B(M,\bk)^*$. 

\begin{proposition} \label{prop:support-typeC}
Choose a partition $\nu \subseteq (k^d)$. The support variety of $\cM^k_\nu$ is reduced and is the variety of skew-symmetric matrices in $\lw^2(E)$ of rank $\le 2k$.
\end{proposition}

\begin{subeqns}
\begin{proof}
We have $\bL(\cM^k_{\nu}) = \cN^k_{\nu^\dagger}$ by Theorem~\ref{thm:main-typeC}. With the notation in \S\ref{sec:howe}, we have
\begin{align} \label{eqn:supp1}
\sU = \Sym(U \otimes (E^*|F)) = \Sym(U \otimes E^*) \otimes \lw^\bullet(U \otimes F).
\end{align}
The first tensor factor has a multiplicity-free action of $\Sp(U) \times \fso(E \oplus E^*)$, and the second has a multiplicity-free action of $\Sp(U) \times \fsp(V)$. We have
\begin{equation} \label{eqn:supp2}
\begin{split} 
\Sym(U \otimes E^*) &= \bigoplus_{\ell(\mu) \le \min(n,k)} \bS_{[\mu]}(U) \otimes M^{\fso}_\mu,\\
\lw^\bullet(U \otimes F) &= \bigoplus_{\lambda \subseteq k \times (n+d)} \bS_{[\lambda]}(U) \otimes \bS_{[((n+d) \times k) \setminus \lambda^\dagger]}(V),
\end{split}
\end{equation}
where $(\det E)^k \otimes M^{\fso}_\mu$ is the module $M_\mu$ in \cite[\S 3.5]{lwood}. So as an $S$-module, $\cN^k_{\nu^\dagger}$ is a direct sum of modules of the form $M_\mu^\fso$. On each module, the action of $S$ factors through $M_\emptyset^\fso$, which is the coordinate ring of the variety of skew-symmetric matrices of rank $\le 2k$.
\end{proof}
\end{subeqns}

\begin{remark}
Using \cite[Theorem 3.9, Corollary 3.12]{avramov}, we get the following statement. Let $W \subset \lw^2(E)$ be a linear subspace of complementary dimension to the space of skew-symmetric matrices of rank $\le 2k$ so that their intersection is $0$. Then $\cM_\nu^k$ has finite projective dimension over $A/(W)$. It would be interesting to understand the minimal free resolution of $\cM^k_\nu$ over $A/(W)$.
\end{remark}

\subsection{Cokernels}

For background on multilinear algebra on chain complexes, see \cite[\S 2.4]{weyman}. Let $\phi \colon E \otimes B(-1) \to V \otimes B$ be the generic matrix. Let $\tilde{V} = (E \oplus E^*) \oplus V$ where $E \oplus E^*$ is given an orthogonal form as in \S\ref{sec:typeC-prelim}. The complex 
\[
\bV \colon 0 \to E \otimes B(-2) \xrightarrow{\phi} V \otimes B(-1) \xrightarrow{\phi^*} E^* \otimes B
\]
carries an action of the orthosymplectic Lie superalgebra $\osp(\tilde{V}) \otimes B$. Here the indexing is such that $\bV_0 = E^* \otimes B$, $\bV_1 = V \otimes B(-1)$, and $\bV_2 = E \otimes B(-2)$. Then we have
\begin{align*}
\Sym^2_B(\bV) \colon 0 \to &\Sym^2(E) \otimes B(-4) \to E \otimes V \otimes B(-3) \to\\ & (\lw^2(V) \oplus (E \otimes E^*)) \otimes B(-2) \to E^* \otimes V \otimes B(-1) \to \Sym^2(E^*) \otimes B.
\end{align*}
The space of $\fgl(E) \times \fsp(V)$ invariants in $E^* \otimes V \otimes B(-1)$ in degree $2$ is $1$-dimensional, and the space of invariants in $\lw^2(V) \oplus (E \otimes E^*)$ is $2$-dimensional, so there is an invariant in the latter space which represents a nonzero homology class. We write this as a map of complexes $B[-2] \to \Sym^2_B(\bV)$. Using the multiplication map for symmetric powers, we get an injective map of complexes $(\Sym^{k-2}_B (\bV))[-2] \to \Sym^k_B (\bV)$, and the cokernel is denoted $\Sym^k_0 (\bV)$. 

\begin{proposition} \label{prop:finite-complex}
If $2\dim(E) + 2d = \dim(V)$, then $\Sym^{d+1}_0 (\bV)$ is acyclic and resolves $\Sym^{d+1}_B(\coker \phi^*)$. In particular, $\Sym^{d+1}_B(\coker \phi^*)$ is a perfect module.
\end{proposition}

\begin{proof}
Set $\bF = \Sym^{d+1}_0 (\bV)$, which is a complex of length $2d+2$. By the Peskine--Szpiro acyclicity lemma \cite[Theorem 20.9]{eisenbud}, $\bF$ is acyclic if and only if the localization $\bF_\fp$ is acyclic for every prime ideal $\fp \subset B$ with $\grade(\fp) < 2d+2$ (recall that the grade of an ideal is the longest regular sequence contained in it). Pick such a prime $\fp$. The ideal generated by the maximal minors has grade $2d+2$: since $B$ is Cohen--Macaulay, the grade of any ideal $I$ is $\dim(B) - \dim(B/I)$; if $I$ is the ideal of maximal minors, then an easy calculation (or see \cite[Theorem 2.2(1)]{littlewoodcomplexes}) shows that this is $2d+2$. So there is at least one maximal minor that is not in $\fp$. After inverting this minor, there is a change of basis so that  $\phi$ becomes $\phi' = \begin{pmatrix} I_n & 0 \end{pmatrix}$ where $I_n$ is an $n \times n$ identity matrix and $0$ is the $n \times (n+2d)$ zero matrix.

We claim that $\bF_\fp$ is exact. Using the special form of $\phi'$, we can decompose $\bV \otimes_B B_\fp$ as the direct sum of a split exact sequence with $(W \otimes B_\fp)[-1]$, where $\dim(W) = 2d$ and $W$ inherits a symplectic form from $V$ (the image of $E$ in $V$ is isotropic since $\bV$ is a complex, and $W = E^\perp / E$). In particular, $\Sym^{d+1}_{B_\fp}(\bV)$ is a direct sum of a split exact complex with $(\lw^{d+1}(W) \otimes B_\fp)[-d-1]$. Similarly, $\Sym^{d-1}_{B_\fp}(\bV)[-2]$ is a direct sum of a split exact complex with $(\lw^{d-1}(W) \otimes B_\fp)[-d-1]$. The induced map on homology $\lw^{d-1}(W) \otimes B_\fp \to \lw^{d+1}(W) \otimes B_\fp$ is multiplication by the symplectic form on $W$, so it is an isomorphism. This implies that $\bF_\fp$ is exact. Finally, $\rH_0(\Sym^{d+1}_0(\bV)) = \rH_0(\Sym^{d+1}(\bV))$, and the latter module is $\Sym^{d+1}_B(\coker \phi^*)$ by right-exactness of symmetric powers (see, for example, \cite[Proposition A2.2(d)]{eisenbud}).
\end{proof}

\begin{remark}
It remains to construct finite length (linear) resolutions supported on lower-order minors. These should be orthosymplectic generalizations of Schur complexes which in turn should be special cases of orthogonal and symplectic analogues of Schubert complexes. See \cite{schubertcomplexes} for Schubert complexes and their connections with double Schubert polynomials, and see \cite[Exercises 6.34--6.36]{weyman} for an analogue of our situation when $B$ is replaced by the polynomial ring.

For $\dim(E) = n$ and $\dim(V) = 2n+2d$, we have that the rank $n-1$, $n-2$, $n-3$, $n-4$ loci have codimensions $2d+2$, $4d+7$, $6d+15$, and $8d+26$, respectively. When this codimension is odd, we cannot use an orthosymplectic Schur functor as we did above.
\end{remark}

\subsection{Examples} \label{sec:typeC-examples}

In general, the resolution of $\cM_\nu^k$, under the action of $\GL(E) \times \Sp(V)$, contains representations with high multiplicity (hence the main advantage of having an irreducible action of the orthosymplectic Lie superalgebra). Now we explain how the simplest examples can be calculated using a few tricks. This amounts to calculating $\cN_{\nu^\dagger}^k$ (after replacing $E$ with $E^*$), which is the $\bS_{[\nu^\dagger]}(U)$-isotypic component of $\sU$. To do this, use the decompositions \eqref{eqn:supp1} and \eqref{eqn:supp2}. When $\nu = \emptyset$, we get (using that representations of $\Sp(U)$ are self-dual)
\[
\bigoplus_{\mu \subseteq \min(n,k) \times (n+d)} \bS_{[((n+d) \times k) \setminus \mu^\dagger]}(V) \otimes M_\mu^{\fso}.
\]
If $k=1$, then $\Sp(U) = \SL(U)$ and $\bS_\lambda(U) = \bS_{[\lambda_1 - \lambda_2]}(U)$. Using Theorem~\ref{thm:cauchy-id}, we get
\[
M_\mu^{\fso} = \bigoplus_{e \ge 0} \bS_{(\mu_1 + e,e)}(E^*).
\]
Finally, we tensor with $(\det E^*)$ because of the way that Howe duality is setup.

Below, we abbreviate $\bS_\lambda(E) \otimes \bS_{[\mu]}(V)$ with the notation $(\lambda; \mu)$.

\begin{example}
Take $\dim(E) = 4$ and $\dim(V) = 8$ with $V$ symplectic. We consider the ideal $\cM_{\emptyset}^1$ generated by the $4 \times 4$ minors. The first few terms of its minimal free resolution over the coordinate ring of the Littlewood variety are 
\footnotesize \begin{align*}
\bF_0 &= (1,1,1,1;1,1,1,1)\\
\bF_1 &= (2,1,1,1;1,1,1,0)\\
\bF_2 &= (2,2,1,1;1,1,1,1) \oplus (3,1,1,1;1,1,0,0)\\
\bF_3 &= (3,2,1,1;1,1,1,0) \oplus (4,1,1,1;1,0,0,0)\\
\bF_4 &= (3,3,1,1;1,1,1,1) \oplus (4,2,1,1;1,1,0,0) \oplus (5,1,1,1;0,0,0,0)\\
\bF_5 &= (4,3,1,1;1,1,1,0) \oplus (5,2,1,1;1,0,0,0)\\
\bF_6 &= (4,4,1,1;1,1,1,1) \oplus (5,3,1,1;1,1,0,0) \oplus (6,2,1,1;0,0,0,0)\\
\bF_7 &= (5,4,1,1;1,1,1,0) \oplus (6,3,1,1;1,0,0,0)\\
\vdots
\end{align*}
\normalsize
The graded Betti table of $\cM_\emptyset^1$ over the polynomial ring (it is not an ideal in this setting) is
\footnotesize \begin{Verbatim}[samepage=true]
       0    1   2   3   4   5  6  7
total: 42 192 462 607 485 257 77 10
    4: 42 192 270 160  35   .  .  .
    5:  .   . 192 447 288  70  .  .
    6:  .   .   .   . 162 160 45  .
    7:  .   .   .   .   .  27 32 10
\end{Verbatim}
\normalsize
Here are the representations for the resolution of the Littlewood ideal (we removed a factor of $\lw^4(E) = \det E$):
\[
\tiny \begin{array}{l|l|l|l|l|l|l|l}
\lw^4_0 V & \lw^3_0 V \otimes E & \lw^2_0 V \otimes S^2 E & V \otimes S^3 E & S^4 E\\
\hline
& & \lw^3_0 V \otimes \lw^3 E & \begin{array}{l} \lw^2_0 V \otimes \bS_{2,1,1} E \\ \lw^4_0 V \otimes \lw^4 E \end{array} & V \otimes \bS_{3,1,1} E & \bS_{4,1,1} E \\
\hline
& & & & \lw_0^2 V \otimes \bS_{2,2,1,1} E & V \otimes \bS_{3,2,1,1} E & \bS_{4,2,1,1} E\\
\hline
& & & & & \lw^2_0 V \otimes \bS_{2,2,2,2} E & V \otimes \bS_{3,2,2,2} E & \bS_{4,2,2,2} E
\end{array}
\]
Compare this with the proof of Theorem~\ref{thm:bott-gen}.
\end{example}

We end with a slightly more complicated example not covered by the above derivation.

\begin{example}
Take $\dim(E) = 3$ and $\dim(V) = 8$ with $V$ symplectic. We consider the ideal $\cM_{(1)}^1$ generated by the $3 \times 3$ minors. The first few terms of its minimal free resolution over the coordinate ring of the Littlewood variety are
\footnotesize \begin{align*}
\bF_0 &= (1,1,1;1,1,1,0)\\
\bF_1 &= (2,1,1;1,1,1,1) \oplus (2,1,1;1,1,0,0)\\
\bF_2 &= (2,2,1;1,1,1,0) \oplus (3,1,1;1,1,1,0) \oplus (3,1,1;1,0,0,0)\\
\bF_3 &= (3,2,1;1,1,1,1) \oplus (3,2,1;1,1,0,0) \oplus (4,1,1;1,1,0,0) \oplus (4,1,1;0,0,0,0)\\
\bF_4 &= (3,3,1;1,1,1,0) \oplus (4,2,1;1,1,1,0) \oplus (4,2,1;1,0,0,0) \oplus (5,1,1;1,0,0,0)\\
\bF_5 &= (4,3,1;1,1,1,1) \oplus (4,3,1;1,1,0,0) \oplus (5,2,1;1,1,0,0) \oplus (5,2,1;0,0,0,0) \oplus (6,1,1;0,0,0,0)\\
&\vdots \qedhere
\end{align*}
\end{example}

\section{Orthogonal Littlewood varieties} \label{sec:orth-case}

Fix $d' \in \{0,1\}$. In this section, we use $\tau_{2k}$ and $\iota_{2k}$ to denote $\tau^\rD_{2k}$ and $\iota^\rD_{2k}$ (see \S\ref{sec:typeD-weyl}) if $d'=0$ and we use them to denote $\tau^\Delta_{2k}$ and $\iota^{\Delta}_{2k}$ (see \S\ref{sec:spin-weyl}) if $d'=1$. For simplicity of exposition, we assume $\ch(\bk)\ne 2$ so that we do not need to distinguish between quadratic forms and orthogonal forms, although with more care one could include this case in some of the results.

\subsection{Preliminaries} \label{sec:typeBD-prelim}

Let $E$ be a vector space of dimension $n$ and let $V$ be an orthogonal space of dimension $2n+2d + d'$ (we assume $d \ge 0$; some of the results below hold even if $d<0$, but we will not need them). Let $A = \Sym(E \otimes V)$. Using the orthogonal form, we have a $\GL(E) \times \bO(V)$-equivariant inclusion
\[
\Sym^2(E) \subset \Sym^2(E) \otimes \Sym^2(V) \subset \Sym^2(E \otimes V).
\]
Let $B = A / (\Sym^2(E))$ be the quotient of $A$ by the ideal generated by $\Sym^2(E)$.

\begin{example}
Let $n=2$ and $d=0$ and $d'=1$, and pick bases for $E$ and $V$. We think of the coordinates in $A$ as the elements in a $2 \times 5$ matrix:
\[
\begin{pmatrix}
x_{1,1} & x_{1,2} & x_{1,3} & x_{1,4} & x_{1,5} \\
x_{2,1} & x_{2,2} & x_{2,3} & x_{2,4} & x_{2,5} 
\end{pmatrix}
\]
If the orthogonal form on $V$ is given by $\omega_V(v, w) = v_1w_2 + v_2w_1 + v_3w_4 + v_4w_3 + v_5w_5$, then $\Sym^2(E)$ is spanned by the $3$ equations
\[
\begin{array}{r}
x_{1,1}x_{2,2} + x_{1,2}x_{2,1} + x_{1,3}x_{2,4} + x_{1,4}x_{2,3} + x_{1,5}x_{2,5},\\
2(x_{1,1}x_{1,2} + x_{1,3}x_{1,4}) + x_{1,5}^2,\\
2(x_{2,1}x_{2,2} + x_{2,3}x_{2,4}) + x_{2,5}^2.
\end{array} \qedhere
\]
\end{example}

We have $\Spec(A) = \hom(E,V^*)$, the space of linear maps $E \to V^*$. We can identify $\Spec(B) \subset \hom(E,V^*)$ with the set of maps $\phi$ such that the composition $E \xrightarrow{\phi} V^* \cong V \xrightarrow{\phi^*} E^*$ is $0$ (here $V \cong V^*$ is the isomorphism induced by the orthogonal form on $V$). Alternatively, $\Spec(B)$ is the subvariety of maps $E \to V^*$ such that the image of $E$ is an isotropic subspace, i.e., the restriction of the orthogonal form to it is the $0$ form.

\begin{proposition} \label{prop:typeBD-basicfacts}
\begin{compactenum}[\rm (a)]
\item If $2d+d'>0$, then $B$ is an integral domain, i.e., the ideal generated by $\Sym^2(E)$ is prime. If $2d+d'=0$, then $\Spec(B)$ has $2$ irreducible components.

\item As a representation of $\GL(E) \times \bO(V)$, we have
\[
B \approx \bigoplus_{\ell(\lambda) \le n} \bS_\lambda(E) \otimes \bS_{[\lambda]}(V).
\]

\item $\bS_\lambda(E) \otimes \bS_{[\lambda]}(V)$ is in the ideal generated by $\bS_\mu(E) \otimes \bS_{[\mu]}(V)$ if and only if $\lambda \supseteq \mu$.

\item Two closed points in $\Spec(B)$ are in the same $\GL(E) \times \bO(V)$ orbit if and only if they have the same rank as a matrix. Each orbit closure is a normal variety with rational singularities. In particular, they are Cohen--Macaulay varieties.

\item The ideal $(\Sym^2(E))$ is generated by a regular sequence. In particular, $B$ is a complete intersection and is a Koszul algebra.
\end{compactenum}
\end{proposition}

\begin{proof}
For (a), (b), (d), and (e), see \cite[Theorems 2.7, 2.10]{littlewoodcomplexes}. (c) follows from the analogous statement for $\Sym(E \otimes V)$ \cite[Corollary 4.2]{dEP}. 
\end{proof}

Let $\OGr(n+d,V)$ be the isotropic Grassmannian of rank $n+d$ isotropic subspaces of $V$ with tautological subbundle $\cR$. Note that when $2d+d'=0$, $\OGr(n+d,V)$ has $2$ connected components, and otherwise it is connected. Let $\nu \subseteq (k^d)$ be a partition. Using \eqref{eqn:det-module} in a relative situation, we get a $\Sym(E \otimes \cR^*)$-module $\sM^k_\nu(E, \cR^*)$.

\begin{proposition} 
The higher sheaf cohomology groups of $\sM^k_\nu(E, \cR^*)$ vanish and
\[
\rH^0(\OGr(n+d,V); \sM^k_\nu(E, \cR^*)) \approx \bigoplus_{\ell(\lambda) \le n} \bS_{(k^n) + \lambda}(E) \otimes \bS_{[(k^{n})+\lambda, k-\nu_d, \dots, k-\nu_1]}(V).
\]
\end{proposition}

\begin{proof}
Similar to Proposition~\ref{prop:cM-vanish}, but use Theorem~\ref{thm:bott-BD}.
\end{proof}

We define 
\begin{equation} \label{eqn:cM-module-BD}
\begin{split}
\cM^k_\nu &= \rH^0(\OGr(n+d,V); \sM^k_\nu(E, \cR^*))\\
&\approx \bigoplus_{\ell(\lambda) \le n} \bS_{(k^n) + \lambda}(E) \otimes \bS_{[((k^n)+\lambda, k-\nu_d, \dots, k-\nu_1)]}(V).
\end{split}
\end{equation}
So $\cM^k_\nu$ is an $A$-module. In fact, the scheme-theoretic image of $\Spec(\Sym(E \otimes \cR^*)) \to \Spec(A)$ is $\Spec(B)$, so $\cM^k_\nu$ is also a $B$-module. 

The following result is the orthogonal version of the main result of this paper. A more detailed version of this theorem is contained in Theorem~\ref{thm:main-typeBD}.

\begin{theorem} \label{thm:Mnu-linear-res-BD}
Assume $\ch(\bk)=0$. The minimal free resolution of $\cM_\nu^k$ over $B$ is linear.
\end{theorem}

Let $\tilde{V}$ be a $\bZ/2$-graded space (from now on, superspace) with $\tilde{V}_0 = E \oplus E^*$ and $\tilde{V}_1 = V$. Define a nondegenerate symplectic form on $E \oplus E^*$ by $\langle (e, \phi), (e', \phi') \rangle = \phi'(e) - \phi(e')$. Then $\tilde{V}$ has a super skew-symmetric bilinear form by taking the direct sum of the symplectic form on $E \oplus E^*$ and the orthogonal form on $V$. Let $\spo(\tilde{V})$ be the orthosymplectic Lie superalgebra in $\fgl(\tilde{V})$ compatible with this form. We have a $\bZ$-grading on $\spo(\tilde{V})$ supported on $[-2,2]$ (we use the symbol $\oplus$ to separate the different pieces of this grading):
\[
\Sym^2(E) \oplus (E \otimes V) \oplus (\fgl(E) \times \fso(V)) \oplus (E^* \otimes V^*) \oplus \Sym^2(E^*).
\]
Let $\fg$ be the positive part of this grading, i.e.,
\[
\fg = (E^* \otimes V^*) \oplus \Sym^2(E^*).
\]
Then $\Sym^2(E^*)$ is central in $\fg$. For $e \otimes v, e' \otimes v' \in E^* \otimes V^*$, we have $[e \otimes v, e' \otimes v'] = \omega_V(v,v') e e'$ where $\omega_V$ is the orthogonal form on $V^* \cong V$. So Proposition~\ref{prop:CI-koszul} implies the following:

\begin{proposition}  
The Koszul dual of $B$ is the universal enveloping algebra $\rU(\fg)$.
\end{proposition}

Let $F \subset V$ be a maximal isotropic subspace and write $V = F \oplus F^* \oplus L$ ($\dim(L) = d'$). Let $\ol{\fu} = ((E^*|F) \otimes L) \oplus \Sym^2(E^*|F) \subset \spo(\tilde{V})$. It inherits a $\bZ$-grading: 
\[
\ol{\fu}_0 = (F \otimes L) \oplus \lw^2(F),  \quad \ol{\fu}_1 = E^* \otimes (L \oplus F),  \quad \ol{\fu}_2 = \Sym^2(E^*).
\]
Let $\fg' \subset \osp(\tilde{V})$ be the subalgebra generated by $\ol{\fu}$ and $\fg$. Then it also has a $\bZ$-grading:
\[
\fg'_0 = (F \otimes L) \oplus \lw^2(F), \quad \fg'_1 = E^* \otimes V^*,  \quad \fg'_2 = \Sym^2(E^*).
\]

\begin{lemma}
We have the following two exact sequences of Lie superalgebras
\begin{align*}
0 \to \fg \to \fg' \to (F \otimes L) \oplus \lw^2(F) \to 0, \qquad 
0 \to \ol{\fu} \to \fg' \to E^* \otimes F^* \to 0.
\end{align*}
\end{lemma}

\begin{proof}
Similar to the proof of Lemma~\ref{lem:lie-exact}.
\end{proof}

\subsection{Howe duality} \label{sec:howe-BD}

In this section, we assume that $\ch(\bk)=0$.

Let $U$ be a $2k$-dimensional orthogonal space. Then $U \otimes \tilde{V}$ is a superspace with even part $U \otimes (E \oplus E^*)$ and odd part $U \otimes V$ and has a super skew-symmetric bilinear form by taking the tensor product of the forms on $U$ and $\tilde{V}$. Let $\spo(U \otimes \tilde{V})$ be the associated orthosymplectic Lie superalgebra. Pick a maximal isotropic subspace $U' \subset U$. Let $E^*|F$ denote the superspace with even part $E^*$ and odd part $F$ and let $0|L$ be the superspace with odd part $L$. Then $(U \otimes (E^*|F)) \oplus (U' \otimes (0|L))$ is a maximal isotropic subspace of $U \otimes \tilde{V}$ and hence we get an oscillator representation 
\[
\sU = \Sym(U \otimes (E^*|F)) \otimes \lw^\bullet(U' \otimes L)
\]
of $\spo(U \otimes \tilde{V})$. Both $\fso(U)$ and $\spo(\tilde{V})$ are subalgebras of $\spo(U \otimes \tilde{V})$ which commute with one another, so $\fso(U) \times \spo(\tilde{V})$ acts on $\sU$. As a representation of $\fso(U)$, $\sU$ is a direct sum of finite-dimensional representations, so we can use the pin group $\Pin(U)$ (this is a double cover of the orthogonal group $\bO(U)$) instead. If $L = 0$, i.e., $d'=0$, then the action of $\Pin(U)$ factors through an action of $\bO(U)$. In fact, there is more data by using the non-connected group $\Pin(U)$ as we now explain. Pick a partition $\lambda$. 
\begin{itemize}
\item If $d'=0$, $\lambda$ is {\bf admissible} if $\lambda_1^\dagger + \lambda_2^\dagger \le 2k$ and we consider representations $\bS_{[\lambda]}(U)$ of $\bO(U)$. For $\lambda_1^\dagger \le k$, this is discussed in \S\ref{sec:schur-functors}. Otherwise, set $\mu_1^\dagger = 2k - \lambda_1^\dagger$ and $\mu_i^\dagger = \lambda_i^\dagger$ for $i>1$. Then $\bS_{[\mu]}(V)$ and $\bS_{[\lambda]}(V)$ differ by a twist of the sign character and we write $\mu = \lambda^\sigma$. Further, define $\ol{\lambda} = \lambda$ if $\lambda_1^\dagger \le k$, otherwise define $\ol{\lambda} = \lambda^\sigma$.

\item If $d'=1$, $\lambda$ is {\bf admissible} if $\ell(\lambda) \le k$ and we consider highest-weight representations $V_{\lambda + \delta_{\pm}}$ where $\delta_+$ and $\delta_-$ are the half-spin weights of $\Spin(U)$ (see \cite[\S 3.1]{littlewoodcomplexes}). Define
\[
V_{\lambda + \delta} = V_{\lambda + \delta_+} \oplus V_{\lambda + \delta_-},
\]
which can be made into an irreducible representation of $\Pin(U)$. 
\end{itemize}
To unify the notation for both cases, we define 
\[
V_{\lambda + d' \delta} = \begin{cases} \bS_{[\lambda]}(U) & \text{if $d'=0$}\\
V_{\lambda + \delta} & \text{if $d'=1$}\end{cases}.
\]

Let $\lambda$ be an admissible partition with $\lambda_{n+1} \le n+d$. Let $\tilde{\cN}^k_\lambda$ be the $V_{\lambda + d'\delta}$-isotypic component of $\sU$ with respect to the action of $\Pin(U)$. We need some of the following facts about the action of $\Pin(U) \times \spo(\tilde{V})$ (we will not use unitarizability, and hence do not define it; we state it only for completeness):

\begin{theorem}
For an admissible partition $\lambda$ with $\lambda_{n+1} \le n+d$, $\tilde{\cN}^k_\lambda$ is  an irreducible, unitarizable, lowest-weight representation of $\spo(\tilde{V})$. We have $\tilde{\cN}^k_\lambda \cong \tilde{\cN}^k_{\lambda'}$ if and only if $\lambda = \lambda'$. As a representation of $\Pin(U) \times \spo(\tilde{V})$, we have a direct sum decomposition 
\[
\sU = \Sym(U \otimes (E^*|F)) \otimes \lw^\bullet(U' \otimes L) = \bigoplus_{\substack{\lambda \mathrm{\,admissible}\\ \lambda_{n+1} \le n+d}} V_{\lambda + d'\delta} \otimes \tilde{\cN}^k_\lambda.
\]
\end{theorem}

\begin{proof}
For $d'=0$, see \cite[\S\S 5.3.4, 5.3.5]{chengwang} and for $d'=1$, see \cite[\S A.1]{CKW}.
\end{proof}

Note that $\fgl(E) \times \fso(V) \subset \spo(\tilde{V})$ acts on $\tilde{\cN}^k_\lambda$ which is a $\bZ$-graded representation with finite-dimensional pieces and the action can be integrated to the group $\GL(E) \times \SO(V)$. However, when $\dim(V)$ is even, we may not be able to extend this action to $\GL(E) \times \bO(V)$. We now make an additional definition
\[
\cN^k_\lambda = \begin{cases} \tilde{\cN}^k_\lambda \oplus \tilde{\cN}^k_{\lambda^\sigma} & \text{if $d'=0$ and $\ell(\lambda)<k$}\\ \tilde{\cN}^k_\lambda & \text{otherwise} \end{cases}.
\]
Then the action of $\GL(E) \times \SO(V)$ on $\cN^k_\lambda$ extends to $\GL(E) \times \bO(V)$.

\begin{theorem} \label{thm:super-lwood-typeBD}
We have a $\fgl(E) \times \fgl(F)$-equivariant isomorphism
\[
\rH_i(\ol{\fu}; \cN^k_\lambda) = \Tor_i^{\rU(\ol{\fu})}(\cN^k_\lambda, \bk) \cong \bigoplus_{\substack{\alpha\\ \tau_{2k}(\alpha) = \lambda \\ \iota_{2k}(\alpha) = i}} \bS_\alpha(E^*|F) \otimes (\det E^*)^k \otimes (\det F^*)^k.
\]
\end{theorem}

\begin{proof}
See \cite[Theorem 5.7]{CKW}. Their notation does not match ours, so alternatively one can use \cite[\S 4]{lwood} if $d'=0$ and \cite{spin-cat} if $d'=1$ (see proof of Theorem~\ref{thm:super-lwood} for comments). 
\end{proof}

\begin{proposition} \label{prop:typeBD-linear-pres}
Let $\lambda$ be a partition with $\ell(\lambda) \le k$ and $\lambda_{n+1} \le n+d$ and set $\nu = (\lambda_1^\dagger, \dots, \lambda_{n+d}^\dagger)$ and $\mu = (\lambda_{n+d+1}^\dagger, \lambda_{n+d+2}^\dagger, \dots)$. Then
\begin{align*}
\rH_0(\fg; \cN^k_\lambda) &= (\det E^*)^k \otimes \bS_{\mu^\dagger}(E^*) \otimes \bS_{[((n+d) \times k) \setminus \nu]}(V) \\
\rH_1(\fg; \cN^k_\lambda) &= (\det E^*)^k \otimes \bS_{(1 + \nu_{n+d}, \mu_2, \mu_3, \dots)^\dagger}(E^*) \otimes \bS_{[(k-\mu_1 + 1, k - \nu_{n+d-1}, \dots, k - \nu_1)]}(V).
\end{align*}
In particular, as a $\rU(\fg)$-module, $\cN^k_\lambda$ is generated in a single degree and has relations of degree $\lambda_{n+d}^\dagger - \lambda_{n+d+1}^\dagger + 1$. 
\end{proposition}

\begin{proof}
Similar to the proof of Proposition~\ref{prop:typeC-linear-pres}. There is one point to highlight though: in Proposition~\ref{prop:par-hom}, we see that when $\dim(V)$ is even and $\lambda_{n+d}>0$, that $\rH_0(\fn; \bS_{[\lambda]}(V))$ is a sum of two irreducible representations of $\fgl(F)$. This does not happen in Proposition~\ref{prop:typeC-linear-pres} and is accounted for by the fact that $\cN^k_\lambda$ is defined to be a sum of two irreducible representations in these cases. Rather than chase signs, we note that the representations $\bS_{[\lambda]^\pm}(V)$ always come together in each homology group since there is an action of $\bO(V)$ on everything.
\end{proof}

\begin{proposition} \label{prop:ss-chase-BD}
If $\lambda_1 \le d$ and $\ell(\lambda) \le k$, then $\cN^k_\lambda$ has a linear free resolution over $\rU(\fg)$. As a representation of $\GL(E) \times \bO(V)$, $\rH_i(\fg; \cN_\lambda^k)$ is a subrepresentation of 
\[
\bigoplus_{\substack{|\gamma| = i\\ \ell(\gamma) \le n}} \bS_{(k^n) + \gamma}(E^*) \otimes \bS_{[((k^n) + \gamma, k - \lambda^\dagger_d, \dots, k - \lambda^\dagger_1)]}(V).
\]
\end{proposition}

\begin{proof}
The proof of the first part is similar to the proof of Proposition~\ref{prop:ss-chase}. See also the comments in the proof of Proposition~\ref{prop:typeBD-linear-pres}. The proof of the second part is similar to the proof of Corollary~\ref{cor:Nlambda-upper}. As an intermediate step, we need to prove an analogue of Lemma~\ref{lem:gamma-bound}. This is almost the same, but the inequality \eqref{eqn:iota-bound1} is replaced by $\iota_{2k}(\alpha) \ge \sum_{j \ge 1} \max(\alpha_j - d - j, 0)$ since the definition of $\iota_{2k}^\rD$ uses $c(R_\lambda)-1$ rather than $c(R_\lambda)$ (\S\ref{sec:typeD-weyl}). This weaker inequality is the one that is used in \eqref{eqn:iota-bound2}, so there is no problem.
\end{proof}

\subsection{Main result and consequences} \label{sec:main-result-typeBD}

We assume that $\ch(\bk)=0$ in this section. The definition of parabolic Verma modules is the same as in \S\ref{sec:main-result-typeC}.

\begin{theorem} \label{thm:main-typeBD}
Let $\lambda$ be a partition with $\ell(\lambda) \le k$ and $\lambda_{n+1} \le n+d$.

\noindent Set $\nu = (\lambda_1^\dagger, \dots, \lambda_{n+d}^\dagger)$ and $\mu = (\lambda_{n+d+1}^\dagger, \lambda_{n+d+2}^\dagger, \dots)$. Then:
\begin{compactenum}[\rm (a)]
\item $\lambda^\dagger_{n+d} = \lambda^\dagger_{n+d+1}$ if and only if $\cN^k_\lambda$ has a linear presentation as a $\rU(\fg)$-module. 
\item If $\lambda^\dagger_{n+d} = \lambda^\dagger_{n+d+1}$, $\bL(\cN^k_\lambda)$ (notation of \S\ref{sec:koszul}) is the first linear strand of a $B$-module $M$ with a compatible $\GL(E) \times \bO(V)$-action and which is uniquely determined by the fact that it is generated in degree $0$ with degree $1$ relations, and satisfies
\begin{align*}
M_0 &= (\det E)^k \otimes \bS_{\mu^\dagger}(E) \otimes \bS_{[((n+d) \times k) \setminus \nu]}(V),\\
M_1 &= (\det E)^k \otimes \bS_{(\mu_1 + 1, \mu_2, \mu_3, \dots)^\dagger}(E) \otimes \bS_{[(k+1-\nu_{n+d}, k-\nu_{n+d-1}, \dots, k-\nu_1)]}(V).
\end{align*}
\end{compactenum}
If furthermore $\lambda_1 \le d$, then $M \cong \cM^k_{\lambda^\dagger}$, and:
\begin{compactenum}[\rm (a)]
\setcounter{enumi}{2}
\item $\bL(\cN^k_{\lambda})$ is an acyclic linear complex and is a minimal free resolution of $\cM^k_{\lambda^\dagger}$ over $B$,
\item $\bR(\cM^k_{\lambda^\dagger})$ is an acyclic linear complex and is a minimal free resolution of $\cN^k_\lambda$ over $\rU(\fg)$. The action of $\fg$ on this complex extends to an action of $\spo(\tilde{V})$ and each $\bR(\cM^k_{\lambda^\dagger})_i$ is a direct sum of parabolic Verma modules.
\end{compactenum}
\end{theorem}

\begin{proof}
Similar to the proof of Theorem~\ref{thm:main-typeC}.
\end{proof}

\begin{conjecture} 
Let $I_\lambda$ be the ideal generated by $\bS_\lambda(E) \otimes \bS_{[\lambda]}(V)$ in $B$. Then the resolution of $I_\lambda$ is a representation of $\spo(\tilde{V})$. Furthermore, each linear strand is irreducible if $\dim(V)$ is odd and otherwise is a sum of $2$ irreducible representations which are interchanged by the outer automorphism of $\spo(\tilde{V})$ (this is related to the conjugation action of $\bO(V)$ on $\SO(V)$).
\end{conjecture}

Now we derive a few consequences of Theorem~\ref{thm:main-typeBD}.

First, we begin with an alternative construction of $\cM^k_\nu$. Consider $\OGr(n,V)$ with vector bundles $\eta = E \otimes \cR^*$, and $\cV^k_\nu = (\det E)^k \otimes (\det \cR^*)^k \otimes \bS_{[(d \times k) \setminus \nu]}(\cR^\perp / \cR)$. 

\begin{proposition}  \label{prop:Mnu-alt-BD}
$\cM^k_\nu = \rH^0(\IGr(n,V); \Sym(\eta) \otimes \cV^k_\nu)$.
\end{proposition}

\begin{proof}
Similar to the proof of Proposition~\ref{prop:Mnu-alt}.
\end{proof}

Define the following vector bundle on $\OGr(n,V)$:
\[
\cE^\nu_\lambda = \bS_\lambda(\cR^\perp) \otimes (\det \cR^*)^k \otimes \bS_{[(d \times k) \setminus \nu]}(\cR^\perp / \cR).
\]

\begin{corollary} \label{cor:Mnu-res-general-BD}
Let $\bF^\nu_\bullet$ be the minimal free resolution of $\cM^k_\nu$ over $A$. We have
\begin{align*} 
\bF^\nu_i = \bigoplus_{j \ge 0} \rH^j(\OGr(n,V); \bigoplus_{\substack{|\lambda| = i+j\\ \lambda \subseteq n \times (n+2d+d')}} \bS_{(k^n) + \lambda}(E) \otimes \cE^\nu_{\lambda^\dagger}) \otimes A(-i-j).
\end{align*}
\end{corollary}

\begin{proof}
Similar to the proof of Corollary~\ref{cor:Mnu-res-general}.
\end{proof}

\begin{theorem} \label{thm:bott-gen-BD}
Let $\nu \subseteq (k^d)$ and $\lambda \subseteq n \times (n+2d+d')$ be partitions and let $\cR$ be the tautological subbundle on $\OGr(n,V)$. 
\begin{compactenum}[\rm (a)]
\item If $k \ge \ell(\lambda)$, then $\rH^i(\OGr(n,V); \cE^\nu_{\lambda^\dagger})=0$ for all $i>0$.
\item The cohomology of $\cE^\nu_{\lambda^\dagger}$ vanishes unless $\tau_{2k}(\lambda) = \mu$ is defined, in which case the cohomology is nonzero only in degree $\iota_{2k}(\lambda) = i$, and we have an $\bO(V)$-equivariant isomorphism
\begin{align*}
\rH^i(\OGr(n,V); \cE^\nu_{\lambda^\dagger}) \cong \rH^0(\OGr(n,V); \cE^\nu_{\mu^\dagger}).
\end{align*}
\end{compactenum}
\end{theorem}

\begin{proof}
Similar to the proof of Theorem~\ref{thm:bott-gen}.
\end{proof}

\begin{remark}
As in \S\ref{sec:typeBD-prelim}, we interpret $A$ as the coordinate ring of the space of linear maps $E \to V^*$ and hence it makes sense to talk about the generic map with respect to $A$, i.e., picking bases for $E$ and $V^*$, the matrix entries are the variables of $A$. Similar comments apply to $B$. The ideal of maximal minors in $B$ of this generic matrix is generated by $\lw^n(E) \otimes \bS_{[1^n]}(V) = \lw^n(E) \otimes \lw^n(V)$. When $\nu = (k^d)$, \eqref{eqn:cM-module-BD} becomes 
\[
\cM^k_{(k^d)} = \bigoplus_{\lambda,\ \lambda_n \ge k} \bS_\lambda(E) \otimes \bS_{[\lambda]}(V),
\]
which is the $k$th power of the maximal minors in $B$ (use Proposition~\ref{prop:typeBD-basicfacts}(c)).
\end{remark}

\begin{corollary} \label{cor:BD-lin-res}
Pick $\nu \subseteq (k^d)$. If $k \ge n$, then the minimal free resolution of $\cM^k_\nu$ over $A$ is linear.
\end{corollary}

\begin{proof}
Combine Corollary~\ref{cor:Mnu-res-general-BD} and Theorem~\ref{thm:bott-gen-BD}(a).
\end{proof}

\begin{proposition} \label{prop:support-typeBD}
Choose a partition $\nu \subseteq (k^d)$. The support variety of $\cM^k_\nu$ is reduced and is the variety of symmetric matrices in $\Sym^2(E)$ of rank $\le 2k$.
\end{proposition}

\begin{proof}
Similar to the proof of Proposition~\ref{prop:support-typeC}. 
\end{proof}

\begin{remark}
When $d=0$, the line bundle $\det \cR^*$ on $\OGr(n,V)$ has a unique square root $\cL$, so it is possible to take powers $(\det \cR^*)^k$ where $k \in \frac{1}{2} \bZ$, and we can define $\cV^k = \cV^k_\emptyset$. The results of this section should mostly hold without change, but the full details of some of the results we cite are not written down when $k \notin \bZ$, so we omit explaining this case in detail. The modules we get this way are submodules of the ``Littlewood spinor module'' studied in \cite[\S 3]{littlewoodcomplexes}, and can be thought of as half-integer powers of the ideal of maximal minors.
\end{remark}

\section{Variety of complexes} \label{sec:gl-case}

In this section, we use $\tau_{2k}$ and $\iota_{2k}$ to mean $\tau^\rA_{2k}$ and $\iota^\rA_{2k}$ (see \S\ref{sec:typeA-weyl}). Given a vector space $V$ and partitions $\lambda, \mu$ with $\ell(\lambda) + \ell(\mu) \le \dim(V)$, let $\bS_{[\lambda; \mu]}(V)$ denote $\bS_{\nu}(V)$ where $\nu = (\lambda_1, \lambda_2, \dots, -\mu_2, -\mu_1)$ (if $\ell(\lambda) + \ell(\mu) < \dim(V)$, there are $0$'s in the $\dots$).

\subsection{Preliminaries} \label{sec:typeA-prelim}

Let $E$ and $F$ be vector spaces of dimensions $n$ and $m$ and let $V$ be a vector space of dimension $n+m+d$ (we assume $d \ge 0$; some of the results below hold even if $d<0$, but we will not need them). Let $A = \Sym((E \otimes V^*) \oplus (V \otimes F^*))$. Throughout, we  fix a decomposition $\dim(V) = a+b$ where $a \ge n$ and $b \ge m$. Using the trace invariant in $V^* \otimes V$, we have a $\GL(E) \times \GL(V) \times \GL(F)$-equivariant inclusion
\[
E \otimes F \subset (E \otimes V^*) \otimes (V \otimes F^*) \subset \Sym^2((E \otimes V^*) \oplus (V \otimes F^*)).
\]
Let $B = A / (E \otimes F^*)$ be the quotient of $A$ by the ideal generated by $E \otimes F^*$.

We have $\Spec(A) = \hom(E,V) \times \hom(V, F)$, and we can identify $\Spec(B) \subset \hom(E,V) \times \hom(V, F)$ with the set of pairs $(\phi, \psi)$ such that $\psi \phi = 0$.

\begin{proposition} \label{prop:typeA-basicfacts}
\begin{compactenum}[\rm (a)]
\item $B$ is an integral domain, i.e., the ideal generated by $E \otimes F^*$ is prime. 

\item As a representation of $\GL(E) \times \GL(V) \times \GL(F)$, we have
\[
B \approx \bigoplus_{\substack{\ell(\lambda) \le n\\ \ell(\lambda') \le m}} \bS_\lambda(E) \otimes \bS_{[\lambda; \lambda']}(V^*) \otimes \bS_{\lambda'}(F^*).
\]

\item $\bS_\lambda(E) \otimes \bS_{[\lambda; \lambda']}(V^*) \otimes \bS_{\lambda'}(F^*)$ is in the ideal generated by $\bS_\mu(E) \otimes \bS_{[\mu; \mu']}(V^*) \otimes \bS_{\mu'}(F^*)$ if and only if $\lambda \supseteq \mu$.

\item Two closed points $(\phi, \psi)$ and $(\phi', \psi')$ in $\Spec(B)$ are in the same $\GL(E) \times \GL(V) \times \GL(F)$ orbit if and only if $\rank(\phi) = \rank(\phi')$ and $\rank(\psi) = \rank(\psi')$. Each orbit closure is a normal variety with rational singularities. In particular, they are Cohen--Macaulay varieties.

\item The ideal $(E \otimes F^*)$ is generated by a regular sequence. In particular, $B$ is a complete intersection and is a Koszul algebra.
\end{compactenum}
\end{proposition}

\begin{proof}
The proofs of (a), (b), (d), and (e) are similar to the proofs of the corresponding statements in Proposition~\ref{prop:typeC-basicfacts}. The desingularization in \cite{littlewoodcomplexes} should be replaced with the bundle $\eta$ on $X$ mentioned directly below. The reader is also encouraged to see \cite{dS} for the general situation of complexes with more than $2$ maps.
(c) follows from the analogous statement for $A$ \cite[Corollary 4.2]{dEP}. 
\end{proof}

\begin{remark}
Part (e) of Proposition~\ref{prop:typeA-basicfacts} also holds if $\dim(V) = \dim(E) + \dim(F) - 1$. One difference is that it has $2$ irreducible components whose general points are pairs $(\phi, \psi)$ where either $\phi$ or $\psi$ has full rank and the other has corank $1$.
\end{remark}

Let $X = \Fl(n, n+d ,V)$ be the partial flag variety of subspaces $R_n \subset R_{n+d} \subset V$ where the subscript indicates the dimension. It has a tautological flag of vector bundles $\cR_n \subset \cR_{n+d} \subset V \times X$. 

Set $\eta = (E \otimes \cR_n^*) \oplus ((V/\cR_{n+d}) \otimes F^*)$. Pick $k, \ell \ge 0$ and partitions $\nu \subseteq (k^a)$, $\nu' \subseteq (\ell^b)$. Set 
\[
\ol{\nu} = (k-\nu_a, \dots, k-\nu_1, \nu'_1 - \ell, \dots, \nu'_b - \ell).
\]
Define
\[
\cV^{k,\ell}_{\nu, \nu'} = (\det E)^k \otimes (\det \cR_n^*)^k \otimes \bS_{\ol{\nu}}((\cR_{n+d}/\cR_n)^*) \otimes (\det V/\cR_{n+d})^\ell \otimes (\det F^*)^\ell.
\]

\begin{proposition} 
The higher sheaf cohomology groups of $\Sym(\eta) \otimes \cV^{k,\ell}_{\nu, \nu'}$ vanish and
\[
\rH^0(X; \Sym(\eta) \otimes \cV^{k,\ell}_{\nu, \nu'}) \approx \bigoplus_{\substack{\ell(\lambda) \le n\\ \ell(\mu) \le m}} \bS_{(k^n) + \lambda}(E) \otimes \bS_{((k^n)+\lambda, \ol{\nu}, -\ell - \mu_m, \dots, - \ell - \mu_1)}(V^*) \otimes \bS_{(\ell^m) + \mu}(F^*).
\]
\end{proposition}

\begin{proof}
Similar to Proposition~\ref{prop:cM-vanish}, but use Theorem~\ref{thm:bott-flag}.
\end{proof}

We define 
\begin{equation} \label{eqn:cM-module-A}
\begin{split}
\cM^{k, \ell}_{\nu, \nu'} &= \rH^0(\Fl(n,n+d,V); \Sym(\eta) \otimes \cV^{k,\ell}_{\nu, \nu'})\\
&\approx \bigoplus_{\substack{\ell(\lambda) \le n\\ \ell(\mu) \le m}} \bS_{(k^n) + \lambda}(E) \otimes \bS_{((k^n)+\lambda, \ol{\nu}, -\ell - \mu_m, \dots, - \ell - \mu_1)}(V^*) \otimes \bS_{(\ell^m) + \mu}(F^*).
\end{split}
\end{equation}
So $\cM^{k,\ell}_{\nu, \nu'}$ is an $A$-module. In fact, the scheme-theoretic image of $\Spec(\Sym(\eta)) \to \Spec(A)$ is $\Spec(B)$, so $\cM^{k,\ell}_{\nu, \nu'}$ is also a $B$-module. 

The following result is the variety of complexes version of the main result of this paper. A more detailed version of this theorem is contained in Theorem~\ref{thm:main-typeA}.

\begin{theorem} \label{thm:Mnu-linear-res-A}
Assume $\ch(\bk)=0$. The minimal free resolution of $\cM^{k,\ell}_{\nu, \nu'}$ over $B$~\mbox{is linear.}
\end{theorem}

Let $\tilde{V}$ be a $\bZ/2$-graded space (from now on, superspace) with $\tilde{V}_0 = E \oplus F$ and $\tilde{V}_1 = V$. We have a $\bZ$-grading on $\fgl(\tilde{V})$ supported on $[-2,2]$ (we use the symbol $\oplus$ to separate the different pieces of this grading):
\[
(E \otimes F^*) \oplus (E \otimes V^* + V \otimes F^*) \oplus (\fgl(E) \times \fgl(V) \times \fgl(F)) \oplus (E^* \otimes V + V^* \otimes F) \oplus (E^* \otimes F).
\]
Let $\fg$ be the positive part of this grading, i.e.,
\[
\fg = (E^* \otimes V + V^* \otimes F) \oplus (E^* \otimes F).
\]
Then $E^* \otimes F$ is central in $\fg$. For $(e \otimes v, \phi \otimes f), (e' \otimes v', \phi' \otimes f') \in E^* \otimes V + V^* \otimes F$, we have 
\[
[(e \otimes v, \phi \otimes f), (e' \otimes v', \phi' \otimes f')] = \phi'(v) e \otimes f' + \phi(v') e' \otimes f.
\]
So Proposition~\ref{prop:CI-koszul} implies the following:

\begin{proposition}  
The Koszul dual of $B$ is the universal enveloping algebra $\rU(\fg)$.
\end{proposition}

Choose a decomposition $V = V^a \oplus V^b$ where $\dim(V^a)=a$ and $\dim(V^b)=b$ (recall that $a \ge n$ and $b \ge m$). Let $\ol{\fu} = (E|V^a)^* \otimes (F|V^b) \subset \fgl(\tilde{V})$. It inherits a $\bZ$-grading: 
\[
\ol{\fu}_0 = {V^a}^* \otimes V^b,  \quad \ol{\fu}_1 = (E^* \otimes V^b) \oplus ({V^a}^* \otimes F),  \quad \ol{\fu}_2 = E^* \otimes F.
\]
Let $\fg' \subset \fgl(\tilde{V})$ be the subalgebra generated by $\ol{\fu}$ and $\fg$. Then it also has a $\bZ$-grading:
\[
\fg'_0 = {V^a}^* \otimes V^b, \quad \fg'_1 = (E^* \otimes V) \oplus (V^* \otimes F), \quad \fg'_2 = E^* \otimes F.
\]

\begin{lemma}
We have the following two exact sequences of Lie superalgebras
\begin{align*}
0 \to \fg \to \fg' \to {V^a}^* \otimes V^b \to 0, \qquad 
0 \to \ol{\fu} \to \fg' \to (E^* \otimes V^a) \oplus ({V^b}^* \otimes F) \to 0.
\end{align*}
\end{lemma}

\begin{proof}
Similar to the proof of Lemma~\ref{lem:lie-exact}.
\end{proof}

\subsection{Howe duality} \label{sec:howe-A}

In this section, we assume that $\ch(\bk)=0$.

Let $U$ be a $(k+\ell)$-dimensional vector space. Then $\GL(U) \times \fgl(\tilde{V})$ acts on the space
\[
\sU = \Sym( (( E|V^a)^* \otimes U^*) \oplus ( U \otimes (F|V^b)) )
\]
extending the obvious action of $\GL(U) \times \fgl(E|V^a) \times \fgl(F|V^b)$ \cite[\S 3.1]{CLZ}.

Let $\lambda, \lambda'$ be partitions with $\ell(\lambda) + \ell(\lambda') \le k+\ell$ and such that $\lambda_{m+1} \le b$ and $\lambda'_{n+1} \le a$. Let $\cN^{k,\ell}_{\lambda, \lambda'}$ be the $\bS_{[\lambda;\lambda']}(U)$-isotypic component of $\sU$ with respect to the action of $\GL(U)$. We need some of the following facts about the action of $\GL(U) \times \fgl(\tilde{V})$ (we will not use unitarizability, and hence do not define it; we state it only for completeness):

\begin{theorem}
Let $\lambda, \lambda'$ be partitions with $\ell(\lambda) + \ell(\lambda') \le k+\ell$ such that $\lambda_{m+1} \le b$ and $\lambda'_{n+1} \le a$. Then $\cN^{k,\ell}_{\lambda, \lambda'}$ is an irreducible, unitarizable, lowest-weight representation of $\fgl(\tilde{V})$. We have $\cN^{k,\ell}_{\lambda, \lambda'} \cong \cN^{k,\ell}_{\mu, \mu'}$ if and only if $\lambda = \mu$ and $\lambda' = \mu'$. As a representation of $\GL(U) \times \fgl(\tilde{V})$, we have a direct sum decomposition 
\[
\sU = \Sym( (( E|V^a)^* \otimes U^*) \oplus ( U \otimes (F|V^b)) ) = \bigoplus_{\substack{\ell(\lambda) + \ell(\lambda') \le k+\ell\\ \lambda_{m+1} \le b\\ \lambda'_{n+1} \le a}} \bS_{[\lambda;\lambda']}(U) \otimes \cN^{k,\ell}_{\lambda, \lambda'}.
\]
\end{theorem}

\begin{proof}
See \cite[Theorem 3.3]{CLZ}.
\end{proof}

\begin{theorem} \label{thm:super-lwood-typeA}
We have a $\fgl(E) \times \fgl(V^a) \times \fgl(V^b) \times \fgl(F)$-equivariant isomorphism
\[
\rH_i(\ol{\fu}; \cN^{k,\ell}_{\lambda, \lambda'}) = \Tor_i^{\rU(\ol{\fu})}(\cN^{k,\ell}_{\lambda, \lambda'}, \bk) \cong \bigoplus_{\substack{(\alpha, \alpha')\\ \tau_{k+\ell}(\alpha, \alpha') = (\lambda, \lambda') \\ \iota_{k+\ell}(\alpha, \alpha') = i}} \bS_{\alpha'}((E|V^a)^*) \otimes \bS_{\alpha}(F|V^b) \otimes \cL
\]
where $\cL = (\det E^*)^k \otimes (\det V^a)^k \otimes (\det {V^b}^*)^\ell \otimes (\det F)^\ell$.
\end{theorem}

\begin{proof}
See \cite[Theorem 5.7]{CKW}. Their notation does not match ours, so alternatively one can use \cite[\S 5]{lwood} (see proof of Theorem~\ref{thm:super-lwood} for comments). The choice of $\cL$ is not unique: one can twist the action of $\fgl(\tilde{V})$ on $\cN^{k,\ell}_{\lambda, \lambda'}$ by the supertrace character $\det \tilde{V} = (\det E) \otimes (\det V^a)^* \otimes (\det V^b)^* \otimes (\det F)$ which does not affect any of its essential properties; \cite{CKW} uses $\cL \otimes (\det \tilde{V})^{-\ell}$ instead of $\cL$.
\end{proof}

\begin{proposition} \label{prop:typeA-linear-pres}
Let $\lambda, \lambda'$ be partitions such that $\ell(\lambda) + \ell(\lambda') \le k+\ell$, $\lambda_{m+1} \le b$, and $\lambda'_{n+1} \le a$. Set $\nu = (\lambda_1^\dagger, \dots, \lambda_b^\dagger)$ and $\mu = (\lambda_{b+1}^\dagger, \lambda_{b+2}^\dagger, \dots)$ and similarly, set $\nu' = ({\lambda'}_1^\dagger, \dots, {\lambda'}_a^\dagger)$ and $\mu = ({\lambda'}_{a+1}^\dagger, {\lambda'}_{a+2}^\dagger, \dots)$. Then
\begin{align*}
\rH_0(\fg; \cN^{k,\ell}_{\lambda, \lambda'}) &= (\det E^*)^k \otimes \bS_{{\mu'}^\dagger}(E^*) \otimes (\det F)^\ell \otimes \bS_{\mu^\dagger}(F) \otimes \bS_{[(k^a) \setminus \nu'; (\ell^b) \setminus \nu]}(V) \\
\rH_1(\fg; \cN^{k,\ell}_{\lambda, \lambda'}) &= (\det E^*)^k \otimes \bS_{(\nu'_a + 1, \mu'_2, \mu'_3, \dots)^\dagger}(E^*) \otimes (\det F)^\ell \otimes \bS_{\mu^\dagger}(F) \otimes \bS_{[(k-\mu'_1 +1, k-\nu'_{a-1}, \dots, k-\nu'_1); (\ell^b) \setminus \nu]}(V) \\
\oplus & (\det E^*)^k \otimes \bS_{{\mu'}^\dagger}(E^*) \otimes (\det F)^\ell \otimes \bS_{(\nu_b, \mu_2, \mu_3, \dots)^\dagger}(F) \otimes \bS_{[(k^a) \setminus \nu'; \ell - \mu_1 + 1, \ell-\nu_{b-1}, \dots, \ell-\nu_1]}(V).
\end{align*}
In particular, as a $\rU(\fg)$-module, $\cN^{k,\ell}_{\lambda, \lambda'}$ is generated in a single degree and has relations of degree $\lambda_b^\dagger - \lambda_{b+1}^\dagger + 1$ and relations of degree ${\lambda'}_a^\dagger - {\lambda'}_{a+1}^\dagger + 1$. 
\end{proposition}

\begin{proof}
Similar to the proof of Proposition~\ref{prop:typeC-linear-pres}. 
\end{proof}

\begin{proposition} \label{prop:ss-chase-A}
Let $\lambda, \lambda'$ be partitions such that $\ell(\lambda) + \ell(\lambda') \le k+\ell$ 
and $\lambda'_1 \le a-n$ and $\lambda_1 \le b-m$. Then $\cN^{k,\ell}_{\lambda, \lambda'}$ has a linear free resolution over $\rU(\fg)$. As a representation of $\GL(E) \times \GL(V) \times \GL(F)$, $\rH_i(\fg; \cN_{\lambda, \lambda'}^{k,\ell})$ is a subrepresentation of 
\[
\bigoplus_{\substack{|\gamma| + |\gamma'| = i\\ \ell(\gamma) \le m\\ \ell(\gamma') \le n}} \bS_{(k^n) + \gamma'}(E^*) \otimes \bS_{[((k^n) + \gamma', k - {\lambda'}^\dagger_{a-n}, \dots, k-{\lambda'}^\dagger_1); ((\ell^m) + \gamma, \ell - \lambda^\dagger_{b-m}, \dots, \ell - \lambda^\dagger_1)]}(V) \otimes \bS_{(\ell^m) + \gamma}(F).
\]
\end{proposition}

\begin{proof}
The proof of the first part is similar to the proof of Proposition~\ref{prop:ss-chase}. See also the comments in the proof of Proposition~\ref{prop:typeA-linear-pres}. The proof of the second part is similar to the proof of Corollary~\ref{cor:Nlambda-upper}. 
\end{proof}

\subsection{Main result and consequences} \label{sec:main-result-typeA}

We assume that $\ch(\bk)=0$ in this section. The definition of parabolic Verma modules is the same as in \S\ref{sec:main-result-typeC}.

\begin{theorem} \label{thm:main-typeA}
Let $\lambda, \lambda'$ be partitions with $\ell(\lambda) + \ell(\lambda') \le k+\ell$ and $\lambda_{m+1} \le b$ and $\lambda'_{n+1} \le a$. Set $\nu = (\lambda_1^\dagger, \dots, \lambda_b^\dagger)$, $\mu = (\lambda_{b+1}^\dagger, \lambda_{b+2}^\dagger, \dots)$, $\nu' = ({\lambda'}_1^\dagger, \dots, {\lambda'}_a^\dagger)$, and $\mu' = ({\lambda'}_{a+1}^\dagger, {\lambda'}_{a+2}^\dagger, \dots)$. Then:
\begin{compactenum}[\rm (a)]
\item $\lambda^\dagger_b = \lambda^\dagger_{b+1}$ and ${\lambda'}^\dagger_a = {\lambda'}^\dagger_{a+1}$ if and only if $\cN^{k,\ell}_{\lambda, \lambda'}$ has a linear presentation as a $\rU(\fg)$-module. 
\item If $\lambda^\dagger_b = \lambda^\dagger_{b+1}$ and ${\lambda'}^\dagger_a = {\lambda'}^\dagger_{a+1}$, then $\bL(\cN^{k,\ell}_{\lambda, \lambda'})$ (notation of \S\ref{sec:koszul}) is the first linear strand of a $B$-module $M$ with a compatible $\GL(E) \times \GL(V) \times \GL(F)$-action and which is uniquely determined by the fact that it is generated in degree $0$ with degree $1$ relations, and satisfies
\begin{align*}
M_0 &= (\det E)^k \otimes \bS_{{\mu'}^\dagger}(E) \otimes (\det F^*)^\ell \otimes \bS_{\mu^\dagger}(F^*) \otimes \bS_{[(k^a) \setminus \nu'; (\ell^b) \setminus \nu]}(V^*), \\
M_1 &= ((\det E)^k \otimes \bS_{(\nu'_a + 1, \mu'_2, \mu'_3, \dots)^\dagger}(E) \otimes (\det F^*)^\ell \otimes \bS_{\mu^\dagger}(F^*) \otimes \bS_{[(k-\mu'_1 +1, k-\nu'_{a-1}, \dots, k-\nu'_1); (\ell^b) \setminus \nu]}(V^*)) \\
\oplus & ((\det E)^k \otimes \bS_{{\mu'}^\dagger}(E) \otimes (\det F^*)^\ell \otimes \bS_{(\nu_b+1, \mu_2, \mu_3, \dots)^\dagger}(F^*) \otimes \bS_{[(k^a) \setminus \nu'; \ell - \mu_1 + 1, \ell-\nu_{b-1}, \dots, \ell-\nu_1]}(V^*)).
\end{align*}
\end{compactenum}
If furthermore $\lambda'_1 \le a-n$ and $\lambda_1 \le b-m$, then $M \cong \cM^{k,\ell}_{{\lambda'}^\dagger, \lambda^\dagger}$, and:
\begin{compactenum}[\rm (a)]
\setcounter{enumi}{2}
\item $\bL(\cN^{k,\ell}_{\lambda, \lambda'})$ is an acyclic linear complex and is a minimal free resolution of $\cM^{k,\ell}_{{\lambda'}^\dagger, \lambda^\dagger}$ over $B$,
\item $\bR(\cM^{k,\ell}_{{\lambda'}^\dagger, \lambda^\dagger})$ is an acyclic linear complex and is a minimal free resolution of $\cN^{k,\ell}_{\lambda, \lambda'}$ over $\rU(\fg)$. The action of $\fg$ on this complex extends to an action of $\fgl(\tilde{V})$ and each $\bR(\cM^{k,\ell}_{{\lambda'}^\dagger, \lambda^\dagger})_i$ is a direct sum of parabolic Verma modules.
\end{compactenum}
\end{theorem}

\begin{proof}
Similar to the proof of Theorem~\ref{thm:main-typeC}.
\end{proof}

\begin{conjecture} 
Let $I_{\lambda, \lambda'}$ be the ideal generated by $\bS_\lambda(E) \otimes \bS_{[\lambda; \lambda']}(V^*) \otimes \bS_{\lambda'}(F^*)$ in $B$. Then the resolution of $I_{\lambda, \lambda'}$ is a representation of $\fgl(\tilde{V})$ and each linear strand is irreducible.
\end{conjecture}

We now derive a few consequences of Theorem~\ref{thm:main-typeA}. Recall that $\ol{\nu} = (k-\nu_a, \dots, k-\nu_1, \nu'_1 - \ell, \dots, \nu'_b - \ell)$. Define 
\[
\cE^{\nu, \nu'}_{\lambda, \lambda'} = (\det \cR^*_n)^k \otimes \bS_{\ol{\nu}}((\cR_{n+d} / \cR)^*) \otimes (\det V/\cR_{n+d})^\ell \otimes \bS_{\lambda^\dagger}((V/\cR_n)^*) \otimes \bS_{{\lambda'}^\dagger}(\cR_{n+d}).
\]

\begin{proposition} \label{prop:Mnu-res-general-A}
Let $\bF^{\nu,\nu'}_\bullet$ be the minimal free resolution of $\cM^{k,\ell}_{\nu,\nu'}$ over $A$. We have
\begin{align*} 
\bF^{\nu,\nu'}_i = \bigoplus_{j \ge 0} \rH^j(\Fl(n,n+d,V); \bigoplus_{\substack{|\lambda| + |\lambda'| = i+j\\ \lambda \subseteq n \times (m+d)\\ \lambda' \subseteq m \times (n+d)}} \bS_{(k^n) + \lambda}(E) \otimes \cE^{\nu, \nu'}_{\lambda^\dagger, {\lambda'}^\dagger} \otimes \bS_{(\ell^m) + \lambda'}(F^*)) \otimes A(-i-j).
\end{align*}
\end{proposition}

\begin{proof}
Similar to the proof of Corollary~\ref{cor:Mnu-res-general}. In the notation of \S\ref{sec:geom}, we have $\eps = ((E \otimes V^*) \oplus (V \otimes F^*)) \otimes \cO_{\Fl(n,n+d,V)}$ and $\xi = (E \otimes (V/\cR_n)^*) \oplus (\cR_{n+d} \otimes F^*)$.
\end{proof}

\begin{theorem} \label{thm:bott-gen-A}
Let $\nu \subseteq (k^a)$ and $\nu' \subseteq (\ell^b)$ be partitions. Also pick partitions $\lambda \subseteq n \times (m+d)$ and $\lambda' \subseteq m \times (n+d)$. Let $\cR_n \subset \cR_{n+d}$ be the tautological flag on $\Fl(n,n+d,V)$. 
\begin{compactenum}[\rm (a)]
\item If $k+\ell \ge \ell(\lambda) + \ell(\lambda') - 1$, then $\rH^i(\Fl(n,n+d,V); \cE^{\nu, \nu'}_{\lambda^\dagger, {\lambda'}^\dagger})=0$ for $i>0$.
\item The cohomology of $\cE^{\nu, \nu'}_{\lambda^\dagger, {\lambda'}^\dagger}$ vanishes unless $\tau_{k+\ell}(\lambda, \lambda') = (\mu, \mu')$ is defined, in which case the cohomology is nonzero only in degree $\iota_{k+\ell}(\lambda, \lambda') = i$, and we have a $\GL(V)$-equivariant isomorphism
\begin{align*}
\rH^i(\Fl(n,n+d,V); \cE^{\nu, \nu'}_{\lambda^\dagger, {\lambda'}^\dagger}) \cong \rH^0(\Fl(n,n+d,V); \cE^{\nu, \nu'}_{\mu^\dagger, {\mu'}^\dagger}).
\end{align*}
\end{compactenum}
\end{theorem}

\begin{remark}
As in \S\ref{sec:typeA-prelim}, we interpret $A$ as the coordinate ring of the space of pairs of linear maps $E \to V$ and $V \to F$ and hence it makes sense to talk about the generic map with respect to $A$, i.e., picking bases for $E, V, F$, the matrix entries are the variables of $A$. Similar comments apply to $B$. When $\nu = (k^a)$ and $\nu' = (\ell^b)$, \eqref{eqn:cM-module-A} becomes 
\[
\cM^{k,\ell}_{(k^a), (\ell^b)} = \bigoplus_{\substack{\lambda, \lambda'\\ \lambda_n \ge k\\ \lambda'_m \ge \ell}} \bS_\lambda(E) \otimes \bS_{[\lambda; \lambda']}(V^*) \otimes \bS_{\lambda'}(F^*),
\]
which is the product of the $k$th power of the maximal minors of $E \to V$ in $B$ with the $\ell$th power of the maximal minors of $V \to F$ (use Proposition~\ref{prop:typeA-basicfacts}(c)).
\end{remark}

\begin{corollary} \label{cor:A-lin-res}
Pick $\nu \subseteq (k^{a-n})$ and $\nu' \subseteq (\ell^{b-m})$. If $k +\ell \ge n+m-1$, then the minimal free resolution of $\cM^{k,\ell}_{\nu, \nu'}$ is linear over $A$. 
\end{corollary}

\begin{proof}
Combine Proposition~\ref{prop:Mnu-res-general-A} and Theorem~\ref{thm:bott-gen-A}(a).
\end{proof}

\begin{proposition} \label{prop:support-typeA}
Choose partitions $\nu \subseteq (k^{a-n})$ and $\nu' \subseteq (\ell^{b-m})$. The support variety of $\cM^{k,\ell}_{\nu, \nu'}$ is reduced and is the variety of matrices in $\hom(E,F)$ of rank $\le k+\ell$.
\end{proposition}

\begin{proof}
Similar to the proof of Proposition~\ref{prop:support-typeC}. 
\end{proof}


\begin{thebibliography}{CKW}

\bibitem[ABW]{ABW} Kaan Akin, David~A. Buchsbaum, Jerzy Weyman, Resolutions of determinantal ideals: the submaximal minors, {\it Adv. in Math.} {\bf 39} (1981), no.~1, 1--30.

\bibitem[AW1]{aw1} Kaan Akin, Jerzy Weyman, Minimal free resolutions of determinantal ideals and irreducible representations of the Lie superalgebra ${\rm gl}(m\vert n)$, {\it J. Algebra} {\bf 197} (1997), no.~2, 559--583.

\bibitem[AW2]{aw2} Kaan Akin, Jerzy Weyman, The irreducible tensor representations of ${\rm gl}(m\vert 1)$ and their generic homology, {\it J. Algebra} {\bf 230} (2000), no.~1, 5--23. 

\bibitem[AW3]{aw3} Kaan Akin, Jerzy Weyman, Primary ideals associated to the linear strands of Lascoux's resolution and syzygies of the corresponding irreducible representations of the Lie superalgebra ${\bf gl}(m\vert n)$, {\it J. Algebra} {\bf 310} (2007), no.~2, 461--490.

\bibitem[Av]{avramov} L.~L. Avramov, Modules of finite virtual projective dimension, {\it Invent. Math.} {\bf 96} (1989), no.~1, 71--101.

\bibitem[AB]{support-ci} Luchezar L. Avramov, Ragnar-Olaf Buchweitz, Support varieties and cohomology over complete intersections, {\it Invent. Math.} {\bf 142} (2000), no.~2, 285--318. 

\bibitem[BGS]{BGS} Alexander Beilinson, Victor Ginzburg, Wolfgang Soergel, Koszul duality patterns in representation theory, {\it J. Amer. Math. Soc.} {\bf 9} (1996), no.~2, 473--527.

\bibitem[Bo]{boffi} Giandomenico Boffi, The universal form of the Littlewood--Richardson rule, {\it Adv. in Math.} {\bf 68} (1988), no.~1, 40--63.

\bibitem[BEH]{BEH} Ragnar-Olaf Buchweitz, David Eisenbud, J\"urgen Herzog, Cohen-Macaulay modules on quadrics, {\it Singularities, representation of algebras, and vector bundles (Lambrecht, 1985)}, 58--116, Lecture Notes in Math. {\bf 1273}, Springer, Berlin, 1987. 

\bibitem[Ch]{chardin} Marc Chardin, Powers of ideals: Betti numbers, cohomology and regularity, {\it Commutative algebra}, 317--333, Springer, New York, 2013.

\bibitem[CKL]{CKL} Shun-Jen Cheng, Jae-Hoon Kwon, Ngau Lam, A BGG-type resolution for tensor modules over general linear superalgebra, {\it Lett. Math. Phys.} {\bf 84} (2008), no.~1, 75--87, \arxiv{0801.0914v2}.

\bibitem[CKW]{CKW} Shun-Jen Cheng, Jae-Hoon Kwon, Weiqiang Wang, Kostant homology formulas for oscillator modules of Lie superalgebras, {\it Adv. Math.} {\bf 224} (2010), no.~4, 1548--1588, \arxiv{0901.0247v2}.

\bibitem[CLZ]{CLZ} Shun-Jen Cheng, Ngau Lam, R. B. Zhang, Character formula for infinite dimensional unitarizable modules of the general linear superalgebra, {\it J. Algebra} {\bf 273} (2004), 780--805, \arxiv{math/0301183v1}.

\bibitem[CW]{chengwang} Shun-Jen Cheng, Weiqiang Wang, {\it Dualities and Representations of Lie Superalgebras}, Graduate Studies in Mathematics {\bf 144}, American Mathematical Society, Providence, RI, 2012.

\bibitem[DEP]{dEP} C.~De Concini, David Eisenbud, C.~Procesi, Young diagrams and determinantal varieties, {\it Invent. Math.} {\bf 56} (1980), no.~2, 129--165. 

\bibitem[DS]{dS} Corrado De Concini, Elisabetta Strickland, On the variety of complexes, {\it Adv. in Math.} {\bf 41} (1981), no.~1, 57--77. 

\bibitem[EN]{eagonnorthcott} J.~A. Eagon, D.~G. Northcott, Ideals defined by matrices and a certain complex associated with them, {\it Proc. Roy. Soc. Ser. A} {\bf 269} (1962), 188--204. 

\bibitem[Ei1]{eisenbud-ci} David Eisenbud, Homological algebra on a complete intersection, with an application to group representations, {\it Trans. Amer. Math. Soc.} {\bf 260} (1980), no.~1, 35--64.

\bibitem[Ei2]{eisenbud} David Eisenbud, {\it Commutative Algebra with a View Toward Algebraic Geometry}, Graduate Texts in Mathematics {\bf 150}, Springer-Verlag, 1995.

\bibitem[Ei3]{syzygies} David Eisenbud, {\it The Geometry of Syzygies. A second course in commutative algebra and algebraic geometry}, Graduate Texts in Mathematics {\bf 229}, Springer-Verlag, New York, 2005. 

\bibitem[FH]{fultonharris} William Fulton, Joe Harris, {\it Representation Theory: A First Course}, Graduate Texts in Mathematics {\bf 129}, Springer-Verlag, New York, 1991.

\bibitem[GLP]{GLP} L.~Gruson, R.~Lazarsfeld, C.~Peskine, On a theorem of Castelnuovo, and the equations defining space curves, {\it Invent. Math.} {\bf 72} (1983), no.~3, 491--506. 

\bibitem[H]{howe-notes} Roger Howe, Perspectives on invariant theory: Schur duality, multiplicity-free actions and beyond, {\it Israel Mathematical Conference Proceedings} {\bf 8}, 1995.

\bibitem[Ja]{jantzen} Jens Carsten Jantzen, {\it Representations of Algebraic Groups}, second edition, Mathematical Surveys and Monographs {\bf 107}, American Mathematical Society, Providence, RI, 2003.

\bibitem[JPW]{JPW} T.~J\'ozefiak, P.~Pragacz, J.~Weyman, Resolutions of determinantal varieties and tensor complexes associated with symmetric and antisymmetric matrices, {\it Young tableaux and Schur functors in algebra and geometry (Toru\'n, 1980)}, 109--189, Ast\'erisque, 87-88, Soc. Math. France, Paris, 1981. 

\bibitem[Ka]{kac} V.~G. Kac, Some remarks on nilpotent orbits, {\it J. Algebra} {\bf 64} (1980), no.~1, 190--213.

\bibitem[Ke]{kempf} George Kempf, On the geometry of a theorem of Riemann, {\it Ann. of Math. (2)} {\bf 98} (1973), 178--185. 

\bibitem[L]{lovett} Stephen Lovett, Resolutions of orthogonal and symplectic analogues of determinantal ideals, {\it J. Algebra} {\bf 311} (2007), no.~1, 282--298. 

\bibitem[M]{macdonald} I.~G.~Macdonald, {\it Symmetric Functions and Hall Polynomials}, second edition, Oxford Mathematical Monographs, Oxford, 1995.

\bibitem[N]{northcott} D.~G. Northcott, {\it Finite Free Resolutions}, Cambridge Tracts in Mathematics {\bf 71}, Cambridge University Press, Cambridge-New York-Melbourne, 1976. 

\bibitem[PW]{PW} Piotr Pragacz, Jerzy Weyman, Complexes associated with trace and evaluation. Another approach to Lascoux's resolution, {\it Adv. in Math.} {\bf 57} (1985), no.~2, 163--207.

\bibitem[S1]{schubertcomplexes} Steven~V Sam, Schubert complexes and degeneracy loci, {\it J. Algebra} {\bf 337} (2011), 103--125, \arxiv{1006.5514v2}.

\bibitem[S2]{saturation} Steven~V Sam, Symmetric quivers, invariant theory, and saturation theorems for the classical groups, {\it Adv. Math.} {\bf 229} (2012), no.~2, 1104--1135, \arxiv{1009.3040v2}.

\bibitem[S3]{derivedsym} Steven~V Sam, Derived supersymmetries of determinantal varieties, {\it J. Commut. Algebra} {\bf 6} (2014), no.~2, 261--286, \arxiv{1207.3309v1}.

\bibitem[S4]{getzler} Steven~V Sam, Homology of analogues of Heisenberg Lie algebras, {\it Math. Res. Lett.} {\bf 22} (2015), no.~4, 1223--1241, \arxiv{1307.1901v1}.

\bibitem[S5]{lwood-minors-appendix} Steven V Sam, Supplement to Orthosymplectic Lie superalgebras, Koszul duality, and a complete intersection analogue of the Eagon--Northcott complex, available as ancillary file at \url{http://arxiv.org/abs/1312.2255}.

\bibitem[SS]{spin-cat} Steven~V Sam, Andrew Snowden, Infinite rank spinor and oscillator representations, \arxiv{1604.06368v1}.

\bibitem[SSW]{lwood} Steven~V Sam, Andrew Snowden, Jerzy Weyman, Homology of Littlewood complexes, {\it Selecta Math. (N.S.)} {\bf 19} (2013), no.~3, 655--698, \arxiv{1209.3509v2}.

\bibitem[SW]{littlewoodcomplexes} Steven~V Sam, Jerzy Weyman, Littlewood complexes and analogues of determinantal varieties, {\it Int. Math. Res. Not. (IMRN)}, no.~13, 4663--4707, \arxiv{1303.0546v3}.

\bibitem[St]{stembridge} John R. Stembridge, Multiplicity-free products and restrictions of Weyl characters, {\it Represent. Theory} {\bf 7} (2003), 404--439.

\bibitem[Wei]{weibel} Charles A. Weibel, {\it An Introduction to Homological Algebra},  Cambridge Studies in Advanced Mathematics {\bf 38}, Cambridge University Press, Cambridge, 1994.

\bibitem[Wey]{weyman} Jerzy Weyman, {\it Cohomology of Vector Bundles and Syzygies}, Cambridge Tracts in Mathematics {\bf 149}, Cambridge University Press, Cambridge, 2003.

\end{thebibliography}
\end{document}